\documentclass[a4paper, 11pt]{article}
\usepackage{amssymb, latexsym, ,amsrefs, amsmath, amsthm, color, marginnote,hyperref,cancel}
\usepackage{a4wide}

\newtheorem{theorem}{Theorem}[section]
\newtheorem{lemma}[theorem]{Lemma}
\newtheorem{proposition}[theorem]{Proposition}
\newtheorem{corollary}[theorem]{Corollary}
\newtheorem{question}[theorem]{Question}
\newtheorem{ex}[theorem]{Example}

\theoremstyle{definition}

\newtheorem{definition}[theorem]{Definition}
\newtheorem{remark}[theorem]{Remark}

\newcommand{\Z}{\mathbb{Z}}
\newcommand{\N}{\mathbb{N}}

\begin{document}
\title{Critical and injective modules over skew polynomial rings}
\author{Ken Brown, Paula A.A.B. Carvalho and Jerzy Matczuk}
\date{\today}
\maketitle

\abstract{Let $R$ be a commutative local $k$-algebra of Krull dimension one, where $k$ is a field.   Let $\alpha$ be a $k$-algebra automorphism of $R$, and define $S$ to be the skew polynomial algebra $R[\theta; \alpha]$. We offer, under some additional assumptions on $R$, a criterion for $S$ to have injective hulls of all simple $S$-modules  locally Artinian - that is, for $S$ to satisfy property $(\diamond)$. It is easy and well known that if $\alpha$ is of finite order, then $S$ has this property, but in order to get the criterion when $\alpha$ has infinite order we found it necessary to classify all cyclic (Krull) critical $S$-modules in this case, a result which may be of independent interest. With the help of the above we show that  $\widehat{S}=k[[X]][\theta, \alpha]$ satisfies $(\diamond)$ for all $k$-algebra automorphisms $\alpha$ of $k[[X]]$.}

\noindent {\it Keywords:} Injective module, Noetherian ring, simple module, skew polynomial ring.
\medskip

\noindent{2010 {\it Mathematics Subject Classification.} 16D50, 16P40, 16S35}

\section{Introduction}\label{intro}
\subsection{Motivating context}\label{context} A Noetherian ring $S$ is said to satisfy property $(\diamond)$ if the injective hull $E_S(V)$ of every simple right or left $S$-module $V$ is locally Artinian, meaning that every finitely generated submodule of $E_S(V)$ is Artinian. This paper forms part of a project whose ultimate aim is to determine when the skew polynomial ring $S :=  R[\theta; \alpha]$ satisfies $(\diamond)$ when $R$ is commutative Noetherian and $\alpha$ is an automorphism of $R$ (so that $\theta r = \alpha(r)\theta$ for $r \in R$). In our earlier paper \cite{BCM} we showed that the key to $(\diamond)$ for such skew polynomial rings is the case where $S$ is primitive. Necessary and sufficient conditions for primitivity are known \cite{JurekLeroy}, and recalled in Theorem \ref{JurekLeroy} below. In particular it is easy to see that $R[\theta; \alpha]$ is never primitive when $|\alpha| < \infty$. For the primitive case we proved in \cite[Theorem 5.4]{BCM}:
\begin{theorem}\label{old} Let $R$ be a commutative Noetherian ring, $\alpha$ an automorphism of $R$. Set $S := R[\theta; \alpha]$ as above, and assume that $S$ is primitive.
\begin{enumerate}
\item[(1)] If $R$ is Artinian then $S$ satisfies $(\diamond)$.
\item[(2)] Suppose that $R$ contains an uncountable field. If the Krull dimension of $R$ is at least 2, or if $\mathrm{Spec}(R)$ is uncountable, then $S$ does not satisfy $(\diamond)$.
\end{enumerate}
\end{theorem}

It's clear that a crucial case resting in the gap between (1) and (2) of Theorem \ref{old} occurs when $R$ is a local Noetherian domain of Krull dimension 1 and $\alpha$ has infinite order. This is the situation we address in the present paper.

\medskip

\subsection{Critical modules}\label{rep} Let $k$ be an arbitrary field.  Let $R$ be a discrete valuation ring which is a $k$-algebra with maximal ideal $M = XR$, with residue field $R/M = k$. Let $\alpha$ be a $k$-algebra automorphism of $R$, so $\alpha (X) = qX$ for some unit $q$ of $R$. Set $S = R[\theta; \alpha]$, and assume that $|\alpha| = \infty$.
Recall that $S$ has Krull dimension 2 \cite[Proposition 6.5.4(i)]{McCR}. Thus to determine the validity or otherwise of $(\diamond)$ for $S$ it is first necessary to study the simple and the 1-critical $S$-modules. This is done in $\S\S$\ref{injective} and $\S\S$\ref{criticals}, using the fact that the localisation $T := S\langle X^{-1} \rangle$ of $S$ is a principal right and left ideal domain, coupled with work of Bavula and Van Oystaeyen \cite{Bavula}. The following result gives an abbreviated version of what we discover. The definitions of the sets {\bf (A)}, {\bf (B)} and {\bf (C)}, of irreducible elements of $S$, can be found in $\S$\ref{taxonomy}. Additionally, set $ \textbf{(F)}  :=  \{f \in k[\theta] : f \textit{ monic irreducible}\}, $ which is a subset of {\bf (C)}. Recall that a nonzero module $M$ with Krull dimension $\mathrm{Kdim}(M)$ is \emph{critical} if $\mathrm{Kdim}(N) < \mathrm{Kdim}(M)$ for all proper factors $N$ of $M$; every nonzero module with Krull dimension contains a critical submodule, \cite[Lemma 6.2.10]{McCR}. We say that two critical $S$-modules $M$ and $N$ are \emph{hull similar} if they contain a nonzero isomorphic submodule. The following result summarises parts of Theorems \ref{Bavula} and \ref{onecrit}.

\begin{theorem}\label{crits} The irreducible elements of $S$ can be divided into 3 disjoint sets {\bf (A)}, {\bf (B)} and {\bf (C)} with the following properties.
\begin{itemize}
\item[(1)] There are bijective correspondences between these sets, together with {\bf (F)}, and the equivalence classes of finitely generated critical torsion right $S$-modules as follows.
\begin{enumerate}
\item[]\qquad {\bf (F)} $\longleftrightarrow \{ V : V \textit{ simple, } \mathrm{dim}_k(V) < \infty \}/\sim.$
\item[]\qquad {\bf (A)} $\longleftrightarrow \{V : V \textit{ unfaithful finitely generated 1-critical}\}/\sim.$
\item[]\qquad {\bf (B)} $\longleftrightarrow \{V : V \textit{ simple, } \mathrm{dim}_k(V) = \infty \}/\sim.$
\item[]\qquad {\bf (C)} $\longleftrightarrow \{V : V \textit{ faithful finitely generated 1-critical}\}/\sim.$
\end{enumerate}
\item[(2)] The equivalence relation $\sim$ is isomorphism for {\bf (F)}, {\bf (A)} and {\bf (B)} and hull similarity for {\bf (C)}.
\item[(3)] In cases {\bf (A)}, {\bf (B)} and {\bf (C)} the map from left to right takes the irreducible element $c$ to the module $S/cS$. In case {\bf (F)} the element $f$ of {\bf (F)} is sent to the module $S/(XS + fS)$.
\end{itemize}
\end{theorem}

\medskip

\subsection{Monoid commutativity}\label{monoid} Armed with the classification of Theorem \ref{crits} we use elementary homological algebra and analysis of non-split extensions of critical $S-$ and $T-$modules to deduce a criterion for $S$ to satisfy $(\diamond)$ in terms of what we call the \emph{monoid commutativity} of the sets {\bf (B)} and {\bf (C)} of irreducible elements of $S$. The following result is part of Theorem \ref{target}.

\begin{theorem}\label{criterion} Let $S$ be as in $\S$\ref{rep}. Then the following are equivalent.
\begin{enumerate}
\item[(1)]  $S$ satisfies $(\diamond)$ for right modules.
\item[(2)] Given irreducible elements $b$ and $c$ of $S$, respectively of types {\bf (B)} and {\bf (C)}, there exist irreducible elements $b'$ and $c'$ of $S$, respectively of types {\bf (B)} and {\bf (C)}, such that
\begin{equation}\label{final} cb \; = \; b'c'.
\end{equation}
\end{enumerate}
\end{theorem}

Naturally there is a parallel result for left modules.

\subsection{$(\diamond)$ for $ \widehat{S} $}\label{clinch} In $\S$\ref{outcome} we specialise from the above setup by taking $R$ to be $k[[X]]$, continuing to assume that $\alpha$ is an arbitrary $k$-algebra automorphism of infinite order. Write $ \widehat{S} := k[[X]][\theta; \alpha]$. We prove a type of converse Eisenstein criterion for $ \widehat{S}$ (Lemma \ref{notirr}), which provides a sufficient condition for an element of $ \widehat{S}$ to be reducible and which may be of independent interest. This permits us to confirm that $ \widehat{S}$ satisfies the condition stated in Theorem \ref{criterion}(2), thus yielding one of the main results of the paper:

\begin{theorem}\label{main}$ \widehat{S}$ satisfies $(\diamond)$.
\end{theorem}

\bigskip

\section{Background results and notation}\label{background}

\subsection{Primitivity of skew polynomial algebras}\label{primitive} We begin by recalling the following definition from \cite{JurekLeroy}:

\begin{definition}\label{special} Given a ring $R$ and $\alpha\in \mathrm{Aut}(R)$, $R$ is \emph{$\alpha$-special} if there is an element $a$ of $R$ such that the following conditions are satisfied.
\begin{enumerate}
 \item For all $n\geq 1$, $N_n^{\alpha}(a):=a\alpha(a)\ldots\alpha^{n-1}(a)\neq 0$.
 \item For every non-zero $\alpha$-stable ideal $I$ of $R$, there exists $n\geq 1$ such that $N_n^{\alpha}(a)\in I$.
\end{enumerate}
When this occurs, the element $a$ is called an \emph{$\alpha$-special element}.
\end{definition}

\noindent Here is the resulting characterisation of primitivity for skew polynomial rings:

\begin{theorem}{\cite[Theorem 3.10]{JurekLeroy}}\label{JurekLeroy}
Let $R$ be a commutative Noetherian ring and let $\alpha \in \mathrm{Aut}(R)$. Then $ R[\theta; \alpha]$ is primitive if and only if $R$ is $\alpha$-special and $\alpha$ has infinite order.
\end{theorem}

\noindent From the definition it follows easily that an $\alpha$-special ring is $\alpha$-prime. Clearly, an $\alpha$-simple ring is $\alpha$-special, with $1$ as $\alpha$-special element in this case.
\noindent Consider, however, the algebra $R$     defined in $\S$\ref{rep}. Clearly  $<X>$ is a proper $\alpha$-ideal
of $R$, so that $R$ is not $\alpha$-simple; but $X$ is an $\alpha$-special element of $R$, so $R$ is $\alpha$-special. Since $\alpha$ is by hypothesis of infinite order, it therefore follows from Theorem \ref{JurekLeroy} that
$$    S  = R[\theta; \alpha]  \textit{  is primitive.} $$

\noindent In fact one can easily confirm by explicit construction that $S$ is primitive, and - even better - the simple $S$-modules can be completely described using work of Bavula and Van Oystaeyen \cite{Bavula}, as we shall explain in $\S\S$\ref{list} and \ref{1-crit}.

\subsection{Property $(\diamond)$}\label{diamond}  Let $R$ be a commutative Noetherian ring and $\alpha$ an automorphism of $R$. In \cite{BCM} the
following question was addressed:
\begin{equation}\label{which} \textit{For which } R \textit{  and } \alpha \textit{ does }   R[\theta ; \alpha ] \textit{ satisfy } (\diamond)?
\end{equation}

It is not difficult to show that the key to this question is to answer it
when $  R[\theta;\alpha]$ is primitive and $R$ is a domain; see \cite[Corollary 3.5 and $\S$4.2]{BCM}.  Bringing Theorem \ref{JurekLeroy} to bear on this situation, we found the following:

\begin{proposition} \cite[Proposition 5.3]{BCM}\label{primprop} Let R be a commutative Noetherian domain and $\alpha$ an automorphism of $R$. Suppose
that $  R[\theta; \alpha]$ is primitive, with $\alpha$-special element $a$.  Let $\cal{A}$ denote the multiplicative subsemigroup of $R\backslash\{0\}$ generated by $\{\alpha^i(a):i\in \Z\}$, so that $\cal{A}$ is an
$\alpha-$invariant multiplicatively closed subset of $R\backslash\{0\}$ and hence satisfies the Ore condition in $R[\theta; \alpha]$. Suppose that
$R{\cal A}^{-1}$ is not a field. Then $R[\theta; \alpha]$ does not satisfy $(\diamond)$.
\end{proposition}

\noindent Notice that Proposition \ref{primprop} fails to determine whether the primitive domains $ S=R[\theta; \alpha]$ of $\S$\ref{rep} satisfy $(\diamond)$. More precisely, we were able to obtain the following necessary and sufficient conditions in \cite{BCM}; the result was stated as Theorem \ref{old}, but is repeated here for the reader's convenience. 

\begin{theorem}\cite[Theorem 5.4]{BCM}\label{primitivecase} Let $R$ be a commutative Noetherian ring, $\alpha$ an automorphism of $R$. Suppose that  $  R[\theta;\alpha]$ is  primitive.
\begin{enumerate}
\item If $R$ has Krull dimension 0 then $ R[\theta;\alpha]$ satisfies $(\diamond)$.
\item Suppose that $R$ contains an uncountable field. Suppose also that either $R$ has Krull dimension at least 2, or $\mathrm{Spec}(R)$ is uncountable. Then $ R[\theta;\alpha]$ does not satisfy $(\diamond)$.
\end{enumerate}
\end{theorem}

\noindent Thus the rings $S$ of $\S$\ref{rep} lie in the gap between the two listed families (a) and (b), and so are fundamental to understanding which skew polynomial extensions of commutative Noetherian rings $R$ satisfy $(\diamond)$. 


\subsection{Notation}\label{notation} For the reader's convenience we repeat here some notation listed earlier. Let $k$ be an arbitrary field. 
 Henceforth  $R$ stands for  a discrete valuation ring which is a $k$-algebra with maximal ideal $M = XR$, with residue field $R/M = k$. Let $\alpha$ be a $k$-algebra automorphism of $R$, so $\alpha (X) = qX$ for some unit $q$ of $R$. Set $S = R[\theta; \alpha]$, and assume that $|\alpha| = \infty$; that is, there does not exist $n \in \mathbb{N}$ such that
\begin{equation}\label{norm} N^{\alpha}_n(q) \; := \; q \alpha (q) \cdots \alpha^{n-1}(q) = 1.
\end{equation}
We write
$$     \mathcal{X} \quad := \quad \{X^n : n \geq 0 \},$$
so $\mathcal{X}$ is an Ore subset of $S$, and we denote by $T$ the localisation of $S$ at $\mathcal{X}$; that is,
$$T \; := \; S\mathcal{X}^{-1} \; = S\langle X^{-1} \rangle =Q[\theta;\alpha],$$
where $Q=Q(R)=R\langle X^{-1} \rangle$ is the quotient field of $R$.
Taking into account that $X$ is a normal element of $S$, elements of $T$ will either be written as $X^{-n}s$ or $s_0X^{-n} $ for some $n\in \N$ and $s, s_0\in S$.
Throughout the paper, ``module'' will mean ``right module'' unless otherwise indicated.

\bigskip

\section{Injective hulls of uniform $ S$-modules}\label{injective}

The class of simple $ S-$modules, and - more generally - the class of uniform $ S-$modules, splits into two families, as outlined in the following two subsections.

\medskip

\subsection{$\mathcal{X}-$torsion $ S-$modules}\label{torsion}

Recall that $ \mathcal{X}$ denotes the Ore set $\{X^n : n \geq 0 \}$ in $ S$.

\begin{lemma}\label{torsion1} Let $E$ be an indecomposable injective $ S-$module, and suppose that the $\mathcal{X}-$torsion submodule of
$E$ is non-zero. Then $E$ is
$\mathcal{X}-$torsion and one of the following cases pertains:
\begin{enumerate}
\item[(1)] $E = E_{ S}( S/X S)$, and this module is an infinite tower of copies of $k(\theta)$.
\item[(2)] $E= E_{ S}( S/(X S + p(\theta) S))$ for some irreducible polynomial $p(\theta) \in k[\theta]$, and $E$ is locally finite dimensional.
\end{enumerate}
\end{lemma}
\begin{proof} By hypothesis there exists a non-zero element $v$ of $E$ with $vX = 0$. Thus $v S$ is a factor of $ S/X S \cong k[\theta]$. Suppose first that $v S \cong k[\theta]$. Then
\begin{equation}\label{extra} E_{ S/X S}(v S) \cong  E_{k[\theta]}(k[\theta]) \cong k(\theta). 
\end{equation}
Now $X$ is a normal element of $ S$ and therefore the ideal $X S$ has the Artin-Rees property \cite[Proposition 4.2.6]{McCR}. This implies that 
$$ E_S(vS) \; = \; \bigcup_{n \geq 1}\mathrm{Ann}_{E_S(vS)}(X^n S). $$
One can now show by induction that $\mathrm{Ann}_{E_S(vS)}(X^n S)$ is a tower of $n$ copies of $k(\theta)$, the case $n = 1$ being given by (\ref{extra}). The induction step follows using the map from $\mathrm{Ann}_{E_S(vS)}(X^{n+1} S)$ to $\mathrm{Ann}_{E_S(vS)}(X^n S)$ given by multiplication by $X$, which preserves the submodule structure and whose kernel is $\mathrm{Ann}_{E_S(vS)}(X S) = k(\theta)$.

Suppose on the other hand $v S$ is a proper factor of $k[\theta]$. Then $vS$ is finite dimensional and hence contains a simple submodule of the form $ S/( X S+ p(\theta) S)$ for some irreducible polynomial $p(\theta)$. Thus $E= E_{ S}( S/(X S + p(\theta) S))$, as stated in (2). Now $E_{ S/X S}( S/(X S + p(\theta) S))$ is Artinian by Matlis's theorem \cite{Matlis}, and the final part of (2) follows from this and the fact that $X S$ has the Artin-Rees property \cite[Proposition 4.2.6]{McCR}, arguing in a similar way to the previous paragraph.
\end{proof}

\medskip

\subsection{$\mathcal{X}-$torsion free uniform $ S-$modules}\label{torsionfree}

The following lemma shows that Lemma \ref{torsion1}(2) encompasses all the unfaithful simple $ S-$modules.



\begin{lemma}\label{both}
Let $V$ be a simple $ S-$module. Then:
\begin{enumerate}
  \item[(1)]  $V$ is faithful if and only if $V$ is  $\mathcal{X}-$torsion free.
  \item[(2)]  If  $V$ is $\mathcal{X}-$torsion free, then $V$  admits a structure as a (necessarily simple) $T-$module.
\end{enumerate}
\end{lemma}

\begin{proof}
(1) Suppose $V$ is faithful. Since $V$ is simple and $X$ is normal, if, for any $n \geq 1$, $X^n$ kills a nonzero element of $V$ then $VX^n = 0$, a contradiction. So $V$ is $\mathcal{X}-$torsion free.

For the reverse implication, suppose that $V$ is a $\mathcal{X}$-torsion free simple $ S$-module, and suppose that $P := \mathrm{Ann}_{ S}(V) \neq \{0\}$. Then $P \cap R = \{0\}$, so $PT$ is a  proper prime ideal of $T$. But $T\langle \theta^{-1} \rangle$ is a simple ring by  \cite[Theorem 1.8.5]{McCR}, and so $\theta \in P$. Since $P$ is a primitive ideal of $ S$, this forces $P$ to be $\theta  S + X S$, contradicting the fact that $V$ is $\mathcal{X}$-torsion free. Therefore $V$ is faithful.

    (2)  Since $X$ is normal and $VX \neq 0$, the $ S-$submodule
$VX$ must equal $V$, so $V$ admits a structure as a $T-$module since $T =
 S\langle X^{-1}\rangle $.
 \end{proof}

\begin{proposition}\label{samehull}\begin{enumerate}
\item[(1)] Let $U$ be an $\mathcal{X}$-torsion free $ S-$module.
Then
$$ E_{ S}(U) = E_T(U \otimes_{ S} T).$$
\item[(2)] Let $V$ be a faithful simple $ S-$module. Then $E_{ S}(V)=E_T(V)$.
\end{enumerate}
\end{proposition}

\begin{proof} (1) If $0 \neq u \otimes_S t = u \otimes_S sX^{-n} \in U \otimes_{ S} T$, with $u \in U, \, s \in  S, \, t \in T$ and $n \in \mathbb{N}$, then $ 0 \neq (u \otimes_{ S} t)X^n \in U$, so that $U \otimes_{ S} T$ is an essential extension of $U$ as
$ S-$modules. The result will therefore follow if we show that
\begin{equation}\label{done} E_{ S}(U \otimes_{ S} T) =
E_T(U \otimes_{ S} T).
\end{equation}
A similar argument to that just given shows that $E_T(U \otimes_{ S} T)$ is an essential extension of $U \otimes_{ S} T$ as $ S$-modules, so that it remains only to show that  $E_T(U \otimes_{ S} T)$ is an injective $ S-$module. For this, let $0\neq  I\lhd_r  S$ and $f\in \mathrm{Hom}_{ S}(I, E_T(U \otimes_{ S} T))$. Define $\overline{f}:IT\longrightarrow E_T(U \otimes_{ S} T)$ such that $\overline{f}(it):=f(i)t$ for $i\in I$ and $t\in T$. It is easy to check that $\overline{f}$ is well-defined, with $\overline{f}_{|I} = f$. So $\overline{f}$ extends to a $T$-homomorphism from $T$ to $E_T(U \otimes_{ S} T)$, and the restriction of this map to $ S$ gives the required extension of $f$ to $ S$. This proves (1).

\medskip

\noindent (2) Let $V$ be a faithful simple $ S$-module. By Lemma
\ref{both} the $ S$-module structure of $V$ extends to a $T$-module structure, with $V = VX$. Hence, $V = V \otimes_{ S} T$,
and so (2) is a special case of (1).
\end{proof}

Let $V$ be a faithful simple $ S$-module. Since $T$ is a prime Noetherian ring of Krull dimension 1, $E_T(V)$ satisfies $(\diamond)$ as a $T$-module (see for instance \cite[Proposition 5.5]{Musson80}). Regarding $(\diamond)$ for $ S$, in the light of Proposition \ref{samehull}(2) and Lemma \ref{torsion1}(2), the remaining issue therefore is:

\begin{question}\label{crux} Given a faithful simple $ S$-module
$V$, a simple $T$-module $W$ which occurs as a subfactor of $E_T(V)$, and
a finitely generated $ S$-submodule $W_0$ of $W$, does $W_0$ have a finite composition series?
\end{question}

\noindent Theorem \ref{main} gives a positive answer to this question in case $R=k[[X]]$.

\bigskip

\section{Irreducible elements}\label{irred}

\subsection{Elementary lemmas}\label{irredlemmas}

\noindent We have seen in $\S$\ref{injective} that the injective hulls of
simple $ S-$ and simple $T-$modules are closely connected. Recall that $T$ is a principal right ideal ring by \cite[Theorem 1.2.9(ii)]{McCR}. Hence, to better understand the representation theory of $ S$ and $T$ we need to study the irreducible elements of these algebras, where, by definition, an \emph{irreducible element} $s$ of a ring $S$ is a
non-zero non-unit of $S$ such that, whenever $s = ab$ with $a,b \in S$,
then either $a$ or $b$ is a unit of $S$.

\begin{lemma}\label{irred} Let $S$ be a noetherian domain. Then every non-zero non-unit of $S$ can be written as a finite product of irreducible elements.
\end{lemma}
\begin{proof} \cite[Proposition 0.9.3]{Cohn}.
\end{proof}

\begin{lemma}\label{irreducible}
\begin{enumerate}
\item[(1)] Let $z \in  S\setminus X S$.  Then $z$ is irreducible in $ S$ if and only if $z$ is irreducible in $T$.
\item[(2)] Let $W$ be a simple $T$-module. Then there is an element $z$ of $ S$ such that $z$ is irreducible in both $ S$ and $T$, with $W\cong T/zT.$
\end{enumerate}
\end{lemma}
\begin{proof} (1) $\Longrightarrow$: Let $z \in  S \setminus X S$, and suppose that $z$ is irreducible in $ S$. Let
\begin{equation}\label{start} z \; = \; (aX^{-\ell})(bX^{-m})
\end{equation}
be a factorisation of $z$ in $T$, with $a,b \in  S \setminus X S$ and $\ell, m \in \mathbb{Z}_{\geq 0}$. Writing $b = \sum_{i=0}^tb_i\theta^i$ with $b_i \in  R$, this yields
$$ z \; = \; a(\sum_i b_i N_i^{\alpha}(q^{\ell})\theta^i)X^{-(\ell + m)},$$
with $a$ and  $\sum_i b_i N_i^{\alpha}(q^{\ell})\theta^i$ in $ S\setminus X S$. Thus
$$ z X^{(\ell + m)} \; = \;  a(\sum_i b_i N_i^{\alpha}(q^{\ell})\theta^i), $$
where the factors appearing in the above equation are all in $ S$, and all except $X^{(\ell + m)}$ are not in $X S$.  Since $X S$ is a completely prime ideal of $ S$, this forces $\ell + m = 0$ and hence
$$ \ell = m= 0.$$
Thus (\ref{start}) is a factorisation of $z$ in $ S$, so one of $a,b$ is a unit in $ S$, as required.

\medskip

\noindent $\Longleftarrow$: Let $z \in  S \setminus X S$, and suppose that $z$ is irreducible in $T$. Suppose that $z = ab$ where $a$ and $b$ are non-zero non-units of $ S$. Thus $a$ or $b$ must be a unit of $T$, say $a$ is such. So
 $$a \in Q \cap  S \; = \;  R, $$
 where $Q=R\langle X^{-1} \rangle$ is the quotient field of $R$.
But $a$ is not a unit of $ S$, thus $a \in X R$. This forces
$z \in X S$, a contradiction. Therefore no such factorisation of
$z$ exists, and $z$ is irreducible in $ S$. If $b$ is a unit in $T$ the argument is similar.

\medskip

(2) Let $W$ be a simple $T$-module. Since $T$ is a principal right ideal ring by \cite[Theorem 1.2.9(ii)]{McCR}, there is an irreducible element $z$ of $T$ such that $W \cong T/zT$. By multiplying $z$ by a suitable power of $X$ we can assume that $z \in  S \setminus X S$, so that $z$ is irreducible  in $ S$ by (1), as required.
\end{proof}

\bigskip

\subsection{Taxonomy}\label{taxonomy}

\noindent Consider the following taxonomy of the irreducible elements of $ S$. Since $ S/X S \cong k[\theta]$, an irreducible element $z$ of $ S$ can be uniquely written in the form
$$ z \, =\,  f + Xs, \qquad \textit{ where } f \in k[\theta], \, s \in  S.$$
After normalising by multiplying $z$ by a suitable unit in $ S$ - that is, by a suitable element from $ R \setminus X R$ - there are the following three mutually exclusive possibilities for $z$:
\begin{eqnarray*}
\textbf{(A)} \qquad z &=& X \textit{ or } z = \theta, \qquad \qquad f
= 0, \, s = 1 \textit{ or } f= \theta, \, s= 0; \\
\textbf{(B)} \qquad z &=& 1 + Xs, \qquad \qquad \qquad s \in  S\setminus  R;\\
\textbf{(C)} \qquad z &=& f + Xs, \qquad \qquad f \in k[\theta]\setminus k,\, f \, \textit{monic}, \, s \in  S, \, z \neq \theta.\\
\end{eqnarray*}
\noindent Note that we forbid $s \in  R$ in type $(\mathbf{B})$ in order to exclude units from the list. We'll make use of these labels below,
where we aim to develop our understanding of the simple and the 1-critical $ S$-modules. We first need an easy lemma:

\begin{lemma}\label{intersect}
Let $z\in S\backslash X S$. Then
\begin{enumerate}
\item[(1)] $ S\cap zT=z S$.
\item[(2)] $ S/z S$ is a $\mathcal{X}$-torsion free $ S-$module.
\end{enumerate}
\end{lemma}
\begin{proof}
(1) Let $t \in T$ with $zt = s_0 \in zT\cap  S$,  with $t=sX^{-n}$ for some $s\in  S$ and $n\in \Z_{\geq 0}$. Thus
\begin{equation}\label{supa} zs=s_0X^n.
\end{equation}
Suppose that $n \geq 1$. Since $X S$ is a completely prime ideal
of the domain $ S$ and $z \notin X S$, (\ref{supa}) shows that $s\in X^n S= SX^n$.  Write $s=\hat{s}X^n$ for some $\hat{s}\in  S$. Then $ t  =  sX^{-n}  =  \hat{s}X^nX^{-n}=\hat{s}\in  S,$ as required.

\noindent (2) Given $s\in S$, if $sX^n\in zS$ for some $n\in \N$, then $s\in  S\cap zT=z S$ by (1).
\end{proof}

\begin{remark}\label{extend}
Lemma \ref{intersect} applies to a product $z$ of type $(\mathbf{B})$ or $(\mathbf{C})$ irreducibles in $ S$.
\end{remark}

\bigskip

\section{Simple and critical $ S-$modules and $T-$modules}\label{criticals}

\subsection{Recap on critical modules}\label{critical}

Denote the (Gabriel-Rentschler) Krull dimension of a module $M$ by $\mathrm{Kdim}(M)$, and recall that our main rings of interest, $ S$ and $T$, have
$$\mathrm{Kdim}( S) = 2 \quad \textit{  and  } \quad \mathrm{Kdim}(T) = 1, $$
by \cite[Theorem 6.5.4(i)]{McCR}. Given a ring $R$ with Krull dimension and a non-negative integer $n$, an $R$-module $M$ is $n$\emph{-critical} if $\mathrm{Kdim}(M) = n$ and $\mathrm{Kdim}(N) < n$ for every proper factor $N$ of $M$. For background on critical modules, see for example \cite[$\S$6.2]{McCR}.

\begin{definition}\label{criticalcomp} Let $R$ be a right noetherian ring, and $M$ a finitely generated (right) $R$-module. A \emph{critical composition series} of $M$ is a finite chain
$$ 0 = M_0 \subset M_1 \subset \cdots \subset M_n = M$$
of submodules of $M$ such that
\begin{enumerate}
\item[$(\bullet)$] $M_i/M_{i-1}$ is critical for $i = 1, \ldots , n$;
\item[$(\bullet)$] $\mathrm{Kdim}(M_i/M_{i-1}) \leq \mathrm{Kdim}(M_{i+1}/M_i)$ for all $i = 1, \ldots , n-1$.
\end{enumerate}
\end{definition}

Finitely generated critical $R$-modules $V$ and $W$ are called \emph{hull-similar} if their injective hulls are isomorphic. Critical modules are uniform, and so critical modules $V$ and $W$ are hull-similar if and only if they share an isomorphic non-zero submodule. It is clear that hull-similarity is an equivalence relation on the class of finitely generated critical $R$-modules. The generalisation of the Jordan-Holder theorem to this setting is as follows:

\begin{theorem}\label{compexist} (\cite[Proposition 6.2.20]{McCR}) Let $R$ be a right noetherian ring, and $M$ a finitely generated (right) $R$-module. Then:
\begin{enumerate}
\item[(1)] $M$ has a critical composition series.
\item[(2)]Any two critical composition series of $M$ have the same length, and after a suitable permutation the composition factors are pairwise hull-similar.
\end{enumerate}
\end{theorem}

\noindent Since $\mathrm{Kdim}( S) =2$, every finitely generated $ S-$module $M$ has $\mathrm{Kdim}(M) \leq 2$ by \cite[Lemma 6.2.5]{McCR}. It's easy to see that the finitely generated 2-critical $ S-$modules are its non-zero right ideals \cite[Proposition 6.3.10]{McCR}, and they form a single hull-similarity class since $ S$ is a domain satisfying the Ore condition. The classification of 0-critical (that is, simple) and 1-critical $ S-$modules runs parallel to the classification of simple $T-$modules, and can therefore exploit the fact that the localisation $T$ of $ S$ is a principal left and right ideal domain,
\cite[Theorem 1.2.9(ii)]{McCR}. We therefore recall in $\S$\ref{lrPID} some classical results about the latter.

\subsection{Principal left and right ideal domains}\label{lrPID}

\begin{definition}\label{similar} (Ore \cite{Ore}) Let $D$ be a left and right PID and let $a,b\in D$.
\begin{enumerate}
\item $a$ is \emph{(right) similar} to $b$ if there exists $u\in D$ such that $1$ is the highest common left factor of $u$ and $b$, (that is, $uD + bD = D$),  and $ua$ is the least common right multiple of $u$ and $b$, (that is, $uaD = uD \cap bD$.)
\item \emph{Left similarity} is defined analogously.
\end{enumerate}
\end{definition}

\noindent The following results can now be proved by straightforward calculation:

\begin{theorem}\label{Jacobson}  Let $D$ be a left and right PID and let $a,b \in D$.
\begin{enumerate}
\item[(1)] ($\mathrm{Ore}$, \cite[Theorem 18]{Ore}, \cite[pages 33-34]{Jacobson}) $a$ and $b$ are right similar if and only if they are left similar (so we may drop the adjectives left and right).
\item[(2)]($\mathrm{Jacobson}$, \cite[Chapter 3, Theorem 4]{Jacobson}) The following are equivalent:
\begin{enumerate}
\item[(i)] $a$ and $b$ are similar.
\item[(ii)] $ D/aD \cong D/bD$.
\item[(iii)] $D/Da \cong D/Db$.
\end{enumerate}
\end{enumerate}
\end{theorem}

\noindent In particular, of course, elements $a$ and $b$ of $D$ are similar if $a = ub$ or $a = bu$ for a unit $u$ of $D$, and this is equivalent to similarity if $D$ is commutative, but not in general.

\medskip

\subsection{Towards a classification of simple $ S-$modules}\label{list}
\noindent Recall the definition of type $\mathrm{{\bf(B)}}$ irreducible elements from $\S$\ref{taxonomy}.

\begin{theorem}\label{Bavula} Let $ S$ and $T$ be as defined in $\S$\ref{notation}.

\begin{itemize}
\item[(1)] Let $V$ be a faithful simple $ S-$module. Then there exists an element $z$ of $ S$ satisfying the following properties (a) and (b).
\begin{enumerate}
\item[(a)] $z$ is irreducible in both $ S$ and $T$;
\item[(b)] As $ S-$modules, $V \cong  S/z S \cong T/zT$.
\end{enumerate}
Moreover, any element $z$ of $ S$ satisfying (b) also satisfies
\begin{enumerate}
\item[(c)] $z$ is a type $\mathrm{{\bf(B)}}$ irreducible.
\end{enumerate}
\item[(2)] Conversely, if $z \in  S$ satisfies $(1)(a)$ and $(1)(c)$, then $ S/z S \cong T/zT$ is a faithful simple $ S-$module.

\item[(3)] Let $z \in  S$ satisfy (1)(a) and (1)(c), and let $y \in  S \setminus X S$. Then the following are equivalent:
\begin{enumerate}
\item[(a)] $ S/z S \cong  S/y S$;
\item[(b)] $T/zT \cong T/yT$ as $T-$modules;
\item[(c)] $z$ and $y$ are similar in $T$.
\end{enumerate}
Moreover, when these equivalent conditions hold $y$ is a type $\mathrm{{\bf(B)}}$ irreducible element of $ S$.

\item[(4)] The simple $ S-$modules which are not faithful are the $\mathcal{X}-$torsion simple modules
$$ V_{f} \quad := \quad  \frac{ S}{X S + f S}, \quad f \in k[\theta], \; f \textit{ a monic irreducible polynomial,}$$
with $V_{f} \cong V_{g}$ if and only if $f = g$.
\end{itemize}
\end{theorem}
\begin{proof}(1):  Let $V$ be a faithful simple $ S-$module. Then $V$ is a simple $T$-module by Lemma \ref{both}. By Lemma \ref{irreducible} there is an element $z \in S$, not equal to $X$ or $\theta$ and irreducible in both $ S$ and $T$, such that $V \cong T/zT$. Since $ S + zT/zT$ is a
non-zero $ S-$submodule of $V$, it must equal $T/zT$, and so it follows using Lemma \ref{intersect} that, as $S$-modules,
$$T/zT \; = \;  S + zT/zT \; \cong \;  S/ S \cap zT  \; = \;  S/z S.$$
This proves (1)$(a)$ and (1)$(b)$.

Suppose now that $z$ is an element of $ S$ satisfying (1)$(b)$. Then $T/zT$ is a simple $T$-module, so $z$ must be irreducible in $T$. Similarly, since $V$ is by hypothesis a simple $ S-$module, $z$ is
irreducible in $ S$ - equivalently, $X$ is not a factor of $z$ in $ S$. To prove (1)$(c)$, we now apply \cite[Theorem 8.1]{Bavula}. This result, specialised to the present setting, states that for an element $z$ of $ S$ which is irreducible in $T$, $T/zT$ is a $\mathcal{X}-$torsion free simple $ S-$module if and only if a condition labelled $\mathrm{(CO)}$ in \cite{Bavula} is satisfied, namely
\begin{equation}\label{CO}  S \; = \; X S + z S.
\end{equation}
Clearly, (\ref{CO}) fails when $z$ is type $(\textbf{A})$ or $(\textbf{C})$, as in these cases $z$ is not congruent to a unit $mod \, X S$. This proves the remainder of (1)$(c)$.

(2) Suppose that $z \in  S$ satisfies (1)$(a)$ and (1)$(c)$. Then (\ref{CO}) holds, so $T/zT$ is a $\mathcal{X}-$torsion free simple $ S-$module by \cite[Theorem 8.1]{Bavula}. This proves (1)$(b)$.

(3) Let $z$ and $y$ be as stated. By part (2), $ S/z S \cong T/zT$ is a faithful simple $ S-$module.

\noindent $(a) \Longrightarrow (b):$ Assume $(a)$. By $(a)$ and part (2),
$$  S/z S \cong T/zT =  S/z S \otimes_{ S} T  \cong   S/y S \otimes_{ S}
T= T/yT,$$
proving $(b)$.

\noindent $(b) \Longrightarrow (c):$ By $(b)$, $y$ is irreducible in $T$,
so $(c)$ follows from Theorem \ref{Jacobson}(2).

\noindent $(c)\Longrightarrow (a):$ Suppose that $z$ and $y$ are similar in $T$. Then $T/zT \cong T/yT$ by Theorem \ref{Jacobson}(2), so in particular $y$ is irreducible in $T$. Thus it is also irreducible in $ S$ by Lemma \ref{irreducible}(1). Hence $ S/z S \cong  S/y S$ by \cite[Theorem 8.1]{Bavula}.

Suppose that $(a)-(c)$ hold. Then $y$ is irreducible in $ S$ by Lemma \ref{irreducible}(1), and it must be type $(\mathbf{B})$ because otherwise $V := S/y S$ is not simple, since in that case $VX \subsetneq V$, contradicting $(1)(a)$.

(4) The first part of (4) follows from Lemma \ref{torsion1}. If $f \neq g$ are monic irreducible polynomials in $k[\theta]$ then $V_{f}$ and $V_{g}$ have different annihilator ideals, so are not isomorphic.
\end{proof}

\begin{remark}\label{gap} There is an obvious missing component in Theorem \ref{Bavula} as a classification of the simple $ S-$modules - namely, it gives no answer to:

\begin{question}\label{whensim} Given type $(\mathbf{B})$ irreducible elements $a$ and $b$ of $ S$, is there a method to determine whether $a$ and $b$ are similar in $T$?
\end{question}

\noindent Beyond the rather cumbersome definition of similarity we note only the observation of Jacobson \cite[page 36]{Jacobson} that, for skew polynomial algebras of automorphism type whose coefficient ring is a division ring, such as $T = R\langle X^{-1}\rangle[\theta; \alpha]$,
$$\textit{ similar polynomials have the same degree in } \theta,$$
since this degree determines the dimension as an $R\langle X^{-1}\rangle-$vector space of the associated cyclic factor module.

\end{remark}

The following easy lemma will be needed in the discussion of 1-critical modules in $\S$\ref{Trep}.

\begin{lemma}\label{split} Let $V$ be a faithful simple $ S-$module and $U$ an unfaithful simple $ S-$module. Then $\mathrm{Ext}^1_{ S}(U,V) = 0.$
\end{lemma}
\begin{proof} Let $V$ and $U$ be as stated, and suppose that $Y$ is an indecomposable extension of $V$ by $U$. By Theorem \ref{Bavula}(4) there exists a monic irreducible polynomial  $f \in k[\theta]$ such that $U = V_{f} =  S/(X S+ f  S)$. In particular, there exists $y \in Y \setminus V$ such that
$yX \in V.$ Moreover, $yX \neq 0$, since $V$ is $\mathcal{X}-$torsion free by Lemma \ref{both}, and $V$ is by hypothesis an essential submodule of
$Y$, so $Y$ is also $\mathcal{X}-$torsion free.

But the $ S-$action on $V$ extends to a structure as $T-$module,
by Lemma \ref{both}. Therefore
$$ y = (yX)X^{-1} \in V, $$
a contradiction. So no such module $Y$ exists.
\end{proof}

\subsection{Hull similarity classes of 1-critical $ S$-modules}\label{1-crit}

\noindent In this section we state and prove an anologue of Theorem \ref{Bavula} for 1-critical $ S$-modules. As preparation for this we need to study certain 1-critical cyclic modules, as follows:

\begin{lemma}\label{typee}  Let $ S$ and $T$ be as defined in $\S$\ref{notation} and let $z$ be a type $\mathrm{(\bf{C})}$ irreducible element of $ S$.
\begin{enumerate}
\item[(1)] Every simple subfactor $V$ of $ S/z S$ is finite dimensinal;
\item[(2)] $ S/z S$ is a faithful cyclic 1-critical $ S$-module.
\end{enumerate}
\end{lemma}

\begin{proof}(1)   Write $W$ for $ S/z S$. Since $z$ is irreducible in $T$ by Lemma \ref{irreducible}(1), $T/zT = W \otimes_{ S} T$ is a simple $T$-module. Since $z$ has type  $\mathrm{(\bf{C})}$, Lemma \ref{intersect} shows
that
\begin{equation}\label{torfree1} W \textit{ is }\mathcal{X}\textit{-torsion free}.
\end{equation}
Let $B$ be a non-zero submodule of $W$, and define $M := W/B$. Since $T$ is a flat left $ S-$module, the short exact sequence $ 0 \longrightarrow B \longrightarrow W \longrightarrow M \longrightarrow 0 $ yields the exact sequence of $T$-modules
\begin{equation}\label{exact} 0 \longrightarrow B \otimes_{ S} T
\longrightarrow W \otimes_{ S} T \longrightarrow M \otimes_{ S} T \longrightarrow 0.
\end{equation}
By (\ref{torfree1}), $B \otimes_{ S} T \neq 0$, so that $B \otimes_{ S} T = W \otimes_{ S} T$ since $W \otimes_{ S} T$ is simple. Therefore, by exactness of (\ref{exact}), $M \otimes_{ S} T = 0$; in other words, $M$ is a finitely generated $\mathcal{X}$-torsion $ S-$module. Since $M$ is finitely generated and $X$ is a normal element of $ S$, there is a positive integer
$\ell$ such that
\begin{equation}\label{kill} MX^{\ell} = 0.
\end{equation}
In particular, if $M$ is non-zero we can choose $0 \neq m \in M$ with $mX
= 0$, say $m = w + B$ for some $w \in W$. Let $I := \mathrm{Ann}_{ S}(w)$, so $I\neq 0$ since $\mathrm{Kdim}(W) \leq 1$ by \cite[Lemma 6.3.9]{McCR}. We claim that
\begin{equation}\label{big} I \nsubseteq  SX.
\end{equation}
For, choose $0 \neq \beta \in I$ and write $\beta$ as $\tau X^r$ for some
$r \geq 0$, with $\tau \notin  SX$.  (We can find $\tau $ since $\bigcap_n ( SX)^n = 0$.) If $w \tau \neq 0$ then the equation
$w \tau X^r = 0$ shows that there exists a non-zero $\mathcal{X}-$torsion element of $W$, contradicting (\ref{torfree1}). This proves (\ref{big}). Thus
$$ m(I + X S) = 0,$$
and hence, by (\ref{big}), $m S$ is a non-zero finite dimensional submodule of $M$. Repeating this argument in the factor of $W$ by $B_1 := B + w S$, and so on, the fact that $W$ is noetherian forces
\begin{equation}\label{size1} \mathrm{dim}_k (W/B) \, < \, \infty.
\end{equation}
We claim next that
\begin{equation}\label{socle} \mathrm{socle}(W) = 0.
\end{equation}
To see this, suppose that $U$ is a simple submodule of $W$. By (\ref{torfree1}) and Theorem \ref{Bavula}(4), $\mathrm{dim}_k (U) = \infty$. Notice that, by Theorem \ref{Bavula}(1),(2) and (3) and the fact that $z$ has
type $(\mathbf{C})$, $W$ itself cannot be simple, so that $U\subsetneq W$. We can therefore apply (\ref{size1}) with $B = U$ to conclude that $W/U$ contains a submodule $A/U$ with $\mathrm{dim}_k(A/U) < \infty$. However Lemma \ref{split} now implies that $A$ splits,
$$ A \quad \cong \quad U \oplus C, $$
where $C$ is a finite-dimensional simple submodule. Since $CX = 0$, this contradicts (\ref{torfree1}). This proves (\ref{socle}).

Now let $B \subseteq A$ be submodules of $W$ with $A/B$ a simple module. Then $B$ must be non-zero by (\ref{socle}).  Therefore (\ref{size1}) implies that $\mathrm{dim}_k (A/B) < \infty$, proving part (1).

\medskip

\noindent (2) It is immediate from  (\ref{socle}) and (\ref{size1}) that $ S/z S$ is a cyclic 1-critical $ S-$module. If $ S/z S$ is not faithful, then it contains a nonzero submodule killed by either $X S$ or by $\theta S$, since these are the only height one primes of $ S$. From (\ref{torfree1}) it is clear that the first of these possibilities is not possible. Suppose that $ S/z S$ has a nonzero submodule killed by $\theta S$. Since $ S/z S$ is critical, it is uniform. Therefore, since $\theta S$ has the Artin-Rees property by \cite[Proposition 4.2.6]{McCR}, there exists $n \geq 1$ such that
\begin{equation}\label{kill} ( S/z S)\theta^n \; = \; 0.
\end{equation}
Let $z = f +Xs$, where $s \in  S$ and $0 \neq f \in k[\theta]$ is monic, with $f \neq \theta$. Then (\ref{kill}) implies that
$$ \theta^n \; = \;  (f + Xs)\gamma $$
for some $\gamma \in  S$, and a quick calculation shows that this is impossible. Hence $ S/z S$ is faithful, as required.
 \end{proof}

Here is the promised 1-critical analogue of Theorem \ref{Bavula}. We repeat Lemma \ref{typee}(2) as part (2) of the theorem, to emphasise it in conjunction with part (1).

\begin{theorem}\label{onecrit} Let $ S$ and $T$ be as defined in $\S$\ref{intro}.

\begin{enumerate}
\item[(1)] Let $M$ be a finitely generated 1-critical  $ S$-module.
\begin{enumerate}
\item[(a)] If $M$ is unfaithful, then $M \cong  S/X S$ or $M \cong  S/\theta  S$.

\item[(b)] If $M$ is faithful, then $M$ is hull similar to $ S/z S$ for a type $\mathrm{(\bf{C})}$ irreducible element $z$ of $ S$. More precisely, $ S/z S$ is isomorphic to a submodule of $M$.
\end{enumerate}

\item[(2)] Let $z$ be a type $\mathrm{(\bf{C})}$ irreducible element of $ S$. Then $ S/z S$ is a faithful cyclic 1-critical $ S$-module.

\item[(3)] Let $z$ and $w$ be irreducible elements of $ S$ with $z$ of type $\mathrm{(\bf {C})}$. Then
the following are equivalent:
\begin{enumerate}
\item[(a)] $ S/z S$ and $ S/w S$ are hull-similar.
\item[(b)] $T/zT$ and $T/wT$ are isomorphic simple $T$-modules.
\item[(c)]  $z$ and $w$ are similar in $T$.
\end{enumerate}
When these equivalent statements hold $w$ is of type $\mathrm{(\bf{C})}$.
\end{enumerate}
\end{theorem}

\begin{proof}(1)$(a)$ Suppose that $M$ is unfaithful. The only height one primes of $ S$ are $X S$ and $\theta S$, so $M$ has a non-zero submodule $A$ with $A(X S) = 0$ or $A(\theta  S) = 0$. Only one of these possibilities can occur for $M$, as otherwise $M$, being uniform, would contain a non-zero submodule killed by $X S + \theta S$, contradicting the fact that $M$ is 1-critical.

So let us suppose that there exists $0 \neq A \subseteq M$ with $A(X S) = 0$; so $\mathrm{Ann}_{ S}(A) = X S$ since $A$ is 1-critical. Since $X S$ satisfies the Artin-Rees property by \cite[Proposition 4.2.6]{McCR}, there exists $t \geq 1$ such that $MX^t = 0 \neq MX^{t-1}$. We claim that
\begin{equation}\label{club} t \; = \; 1.
\end{equation}
Suppose that (\ref{club}) is false. Then we can choose $m \in M$ with
$$ mX^2 \; = 0 \; \neq mX; $$
moreover, since $m S$ is a non-zero submodule of $M$, it is 1-critical and hull similar to $ S/X S$.

Now $m S/m SX$ is Artinian and annihilated by $X S$. Hence there exists $0 \neq f \in k[\theta]$ such that $m Sf \subseteq m SX$. That is, for some non-zero element $\widehat{f}$ of $k[\theta]$,
\begin{equation}\label{park} m SfX \; = \; 0 \; = m SX\widehat{f},
\end{equation}
where the second equality follows from the relation $\theta X = qX\theta$. But now $m SX$ is a non-zero submodule of $M$ which is killed by the co-Artinian ideal $X S +\widehat{f} S$ of $ S$, contradicting the fact that $M$ is 1-critical.

Thus (\ref{club}) is proved. Hence $M$ is a finitely generated torsion-free 1-critical (and so torsion-free and uniform) $k[\theta]$-module. So it is isomorphic to a non-zero ideal of $k[\theta]$, that is to $k[\theta]$ itself.

If $M$ on the other hand contains a non-zero element annihilated by $\theta$, then the argument is similar.

\medskip

(1)$(b)$ Suppose that $M$ is a finitely generated faithful 1-critical $ S-$module. We show first that
\begin{equation}\label{torfree} M \textit{ is } \mathcal{X}-\textit{torsion free.}
\end{equation}
If (\ref{torfree}) is false then there exists $0 \neq m \in M$ with $mX =
0$. But $M$, being critical, is uniform, and is therefore an essential extension of $m S$. Since $(m S)(X S) = 0$ and
the invertible ideal $X S$ has the Artin-Rees property \cite[Proposition 4.2.6]{McCR}, $MX^{\ell}= 0$ for some positive integer $\ell$.
This contradicts the fact that $M$ is faithful, so (\ref{torfree}) is proved.

Now choose $0 \neq m \in M$ and $z \in  S \setminus X S$ with $mz = 0$ and $z$ irreducible in $ S$ and $T$. Note that
we can do this thanks to (\ref{torfree}) coupled with Lemmas \ref{irred} and \ref{irreducible}(1). Moreover $z$ is not type $(\mathbf{A})$ since if it were then $mX = 0$ or $m\theta = 0$, and in both cases the Artin-Rees property applied to the essential extension $m S \subseteq M$ implies that $M$ is unfaithful, a contradiction. Nor is $z$ type $(\mathbf{B})$,
since if it were then $ S/z S$ would be simple by Theorem \ref{Bavula}(2), contradicting the 1-criticality of $M$. Therefore
\begin{equation}\label{type} z \textit{ has type }(\textbf{C}).
\end{equation}
Now $z S \subseteq \mathrm{Ann}_{ S}(m)$, so that $m S$ is a factor of $ S/z S$. But $ S/z S$ is 1-critical by (\ref{type}) and Lemma \ref{typee}(2), and the socle of $M$ is $\{0\}$ since $M$ is 1-critical by hypothesis. So we must have
\begin{equation}\label{inside}   S/z S \cong m S \subseteq M.
\end{equation}
Thus $E_{ S}( S/z S) \cong E_{ S}(M)$ since $M$ is uniform, so $M$ is hull-similar to $ S/z S$ as claimed.

\medskip
(2) This is Lemma \ref{typee}(2).

\medskip

(3) Statements $(b)$ and $(c)$ are equivalent by Theorem \ref{Jacobson}(2).

\noindent $(a)\Longleftrightarrow(b)$: Note first that $ S/z S$ is $\mathcal{X}-$torsion free by Lemma \ref{intersect}(2). Now we have the following chain of equivalences,
 \begin{align*}  S/w S \textrm{ is hull-similar to } S/z S &\Longleftrightarrow& E_{ S}( S/w S) \; &\cong\; E_{ S}( S/z S) \textrm{ as } S-\textrm{modules}\\
&\Longleftrightarrow& E_T( S/w S \otimes_{ S} T) \; &\cong\;  E_T( S/z S \otimes_{ S} T) \textrm{ as } T-\textrm{modules}\\
&\Longleftrightarrow& E_T(T/wT) \; &\cong\;  E_T(T/zT) \textrm{ as } T\textrm{-modules}\\
&\Longleftrightarrow& T/wT \; &\cong\; T/zT \textrm{ as } T\textrm{-modules},
\end{align*}
where the first equivalence is by definition of hull-similarity, the second by Proposition \ref{samehull}(1), and the last equivalence holds because $T/wT$ and $T/zT$ are simple modules due to $T$ being a principal right ideal domain.

Finally, if $(a)$ holds, then $ S/w S$ is faithful and has Krull dimension 1 by part (2) of the theorem applied to $ S/z S$. So $w$ cannot be type $\mathrm{(\bf{A})}$ or type $\mathrm{(\bf{B})}$ by statement (1) and Theorem \ref{Bavula}(2). Hence $w$ must be type $\mathrm{(\bf{C})}$.
\end{proof}

\medskip

In view of the above result we introduce the obvious terminology: a finitely generated critical $ S$-module is \emph{type $\mathrm{(\bf{C})}$ 1-critical} if it is hull-similar to $ S/z S$ where $z$ is a type $\mathrm{(\bf{C})}$ irreducible in $ S$.

\bigskip

\section{$(\diamond)$ for $ S$ via the representation theory of $T$.}\label{Trep}

\subsection{Injective hulls of type $(\mathbf{B})$ $T-$modules}\label{Binj}In this subsection we make use of the analysis of $\S \S$ \ref{list} and \ref{1-crit}  to study injective hulls and extensions of simple $ S-$modules.

We begin by examining the injective hull of a faithful simple $ S-$module $V$, bearing in mind that $V$ admits a structure as $T-$module by Lemma \ref{both}, and its $ S-$ and $T-$hulls are one and the same by Proposition \ref{samehull}(2).

\begin{proposition}\label{torsion} Let $ S$ and $T$ be as in $\S$\ref{notation},
and let $V$ be a faithful simple $ S-$module. Let $\mathcal{C} := \{1 + Xs : s \in  S\}$.
\begin{enumerate}
\item[(1)] $\mathcal{C}$ is an $\alpha$-invariant Ore set in $ S$.
\item[(2)] The set $W$ of $\mathcal{C}$-torsion elements of $E_T(V)$  is a non-zero $T$-submodule of $E_T(V)$.
\item[(3)] Let $c \in \mathcal{C}$, and let $c = c_1 \cdots c_t$ be a factorisation of $c$ as a product of irreducible elements of $ S$. Then $c_i$ has type $\mathrm{{\bf (B)}}$ for all $i$.
\item[(4)] As an $ S-$module, $W$ is locally Artinian, with all composition factors type {\bf (B)}.
\item[(5)] If $E_T(V)/W$ is nonzero  then it has the following properties.
\begin{enumerate}
\item[(i)]  $E_T(V)/W$ is a $\mathcal{C}$-torsion-free injective $T$-module.
\item[(ii)] As $ S-$module $E_T(V)/W$ is injective, a direct sum of injective hulls of 1-critical type {\bf (C)} $ S-$modules.
\item[(iii)] The decomposition $(ii)$ is also a decomposition of $E_T(V)/W$ as a direct sum of injective hulls of simple $T$-modules.
\end{enumerate}
\end{enumerate}
\end{proposition}
\begin{proof}(1) Since $X$ is normal in $ S$, $X S$ satisfies the Artin-Rees property by \cite[Proposition 4.2.6]{McCR}. Hence, by
\cite[Proposition 4.2.9(i)]{McCR}, $\mathcal{C}$ is an Ore set in $ S$. It's clear from the definitions of $\mathcal{C}$ and $\alpha$ that $\mathcal{C}$ is $\alpha$-invariant.

\medskip
\noindent (2) That $W$ is an $ S-$submodule of $E_T (V)$ is a standard consequence of the Ore condition. By Theorem \ref{Bavula}(1) $V \cong T/zT$ for an irreducible element $z$ of $\mathcal{C}$, so that $0 \neq 1 + zT \in W$.

Let $y\in W$ and $c\in\mathcal{C}$ be such that $yc=0$. One can easily see that there is ${\overline c}\in \mathcal{C}$ such that $Xc={\overline c}X$. Then 
$$0=yc=(yX^{-1})({\overline c}X)$$
and $yX^{-1}{\overline c}=0$. Hence $WX^{-1}\subseteq W$ and the result follows. 

\medskip

\noindent (3) This is a straightforward check.

\medskip

\noindent (4) Clearly it is enough to prove that if $w \in W$ then $w S$ is Artinian with all composition factors type $\mathbf{(B)}$. But such an element $w$ satisfies $wc = 0$ for some $c \in \mathcal{C}$, so that $w S$ is a factor of $ S/c S$. The result follows since $ S/c S$ has the desired form by part (3), Lemma \ref{irred} and Theorem \ref{Bavula}(2).

\medskip

\noindent (5) That $E_T(V)/W$ is $\mathcal{C}$-torsion-free is an immediate consequence of the fact that $\mathcal{C}$ is multiplicatively closed.
It is $T$-injective since it is a factor of an injective $T$-module and $T$ is hereditary (being a principal right ideal ring). Since $E_T(V)/W$ is a $T$-module it is $\mathcal{X}$-torsion free. An easy lemma shows that injective $T$-modules are injective as $ S-$modules, so $E_T(V)/W$ is an injective $ S-$module. Since
$ S$ is noetherian, every injective $ S-$module is a direct sum of injective hulls of critical $ S-$modules. So to confirm the claimed structure of $E_T(V)/W$ as $ S-$module we have to confirm that each finitely generated critical $ S-$submodule $U$ of $E_T(V)/W$ contains a cyclic critical of type {\bf (C)}. First note that such a $U$ must be 1-critical: for $U$ is $\mathcal{X}$-torsion free since $E_T(V)/W$ is a $T$-module, so $U$ cannot be finite dimensional; and if $U$ is an infinite dimensional simple $ S-$module then it is type {\bf (B)}, and hence $\mathcal{C}$-torsion, which contradicts the fact that $E_T(V)/W$ is $\mathcal{C}$-torsion free. But now Theorem \ref{onecrit}(1) shows that $U$ is either ${\widehat S}/\theta{\widehat S}$ or it contains a submodule $ S/z S$ for a type {\bf (C)} irreducible $z$.

The first of these possibilities is ruled out by an argument similar to that used in part (2). Namely, if $U \cong  S/\theta S$ then there is an element $e$ of $E_T(V) \setminus W$ with $e\theta \in W$. But then $e\theta d = 0$ for an element $d$ which is a product of type {\bf(B)} irreducibles. It is easy to check that $\theta d = d'\theta$ where $d' \in \mathcal{C}$. Since $E_T(V)$ is an essential extension of the faithful simple $ S$-module $V$, it contains no non-zero elements killed by $\theta$. Hence $ed' = 0$, so $e \in W$, a contradiction.

The final claim in (5) follows immediately from Proposition \ref{samehull}(1).
\end{proof}

\begin{corollary}\label{yip} Continue with the notation of Proposition \ref{torsion}. Then the following are equivalent.
\begin{enumerate}
\item[(1)] $E_{ S}(V)$ is a locally Artinian $ S-$module.
\item[(2)] $W = E_{ S}(V)$.
\item[(3)] $W = E_T(V)$.
\end{enumerate}
\end{corollary}
\begin{proof}$(2) \Leftrightarrow (3)$: By Proposition \ref{samehull}.

$(2) \Leftrightarrow (1)$: By Proposition \ref{torsion}(4),(5).
\end{proof}

\bigskip

\subsection{Equivalent conditions for $(\diamond)$ for $ S$}\label{equivs}
We summarise the equivalences between the nature of $\mathrm{Ext}$-spaces
of critical modules for $ S$ and $T$ and the validity of property $(\diamond)$ for $ S$ in Theorem \ref{gather}, which is a consequence of the following well-known result.

\begin{proposition}\cite[Corollary 3.2]{PP_99}\label{PP}
Let $U$ be a Noetherian domain, $a,b$ nonzero elements of $U$. Then the following are equivalent:
\begin{enumerate}
\item[(a)] $\mathrm{Ext}^1_U(U/aU, U/bU) = 0$;
\item[(b)] $\mathrm{Ext}^1_U(U/Ub, U/Ua) = 0$;
\item[(c)] $Ua + bU = U$.
\end{enumerate}
\end{proposition}

Before applying Proposition \ref{PP} to the analysis of $(\diamond)$ for $ S$ we record the following corollary, which we will need below and which can be deduced either from the Artin-Rees property of $\theta T$ together with $(a) \Longleftrightarrow (b)$, or by elementary calculations in $T$ using criterion $(c)$ along with our taxonomy in $\S$\ref{taxonomy} of canonical irreducible elements.

\begin{corollary}\label{heck} Let $c$ be an irreducible element of $T$ which is not an associate of $\theta$. Then
$$ \mathrm{Ext}^1_T(T/Tc, T/T\theta) \; = \; \mathrm{Ext}^1_T(T/T\theta,T/Tc) \; = \; 0,$$
and similarly for right modules.
\end{corollary}

The following lemma is needed for the proof of Theorem \ref{gather}

\begin{lemma}\label{patch} Let $w$ and $z$ be irreducible elements of $ S$, with $w$ of type {\bf (B)} and $z$ of type {\bf (C)}. Then $\mathrm{Ext}^1_S(S/zS,S/wS) \neq 0$ if and only if there is a short exact sequence of $S$-modules
$$ 0 \longrightarrow S/wS \longrightarrow A \longrightarrow S/zS \longrightarrow 0 $$
with $A$ uniform.
\end{lemma}

\begin{proof} The implication from right to left is clear from basic properties of $\mathrm{Ext}$. For the reverse implication, set $W :=  S/w S$ and $U :=  S/z S$, and suppose that $\mathrm{Ext}^1_S(U,W) \neq 0$. Then there exists a module extension
\begin{equation}\label{this} 0 \longrightarrow W \xrightarrow{\alpha} A \xrightarrow{\pi} U \longrightarrow 0 
\end{equation}
in which $A$ does not split as the direct sum of $U$ and $W$. We suppose for a contradiction that $A$ is not uniform. Thus $\alpha (W)$ is not essential in $A$, so we can choose a non-zero submodule $B$ of $A$ maximal such that 
\begin{equation}\label{grr} B \,\cap \, \alpha (W) \; = \; 0.
\end{equation}
We now obtain from (\ref{this}) the exact sequence
\begin{equation}\label{that}  0 \longrightarrow W \xrightarrow{\overline{\alpha}} A/B \longrightarrow U/\pi(B) \longrightarrow 0,
\end{equation}
Note that, in (\ref{that}), 
\begin{equation}\label{huh} U/\pi(B) \; \neq \; 0 \; \neq \pi(B), 
\end{equation}
due to (\ref{grr}) and the assumption that (\ref{this}) is non-split.

We claim that $\overline{\alpha} (W)$ is essential in $A/B$. Suppose that this is false, and let $C$ be a non-zero submodule of $A/B$ with $\overline{\alpha} (W) \cap C = 0$. Writing $C$ as $D/B$ for a submodule $D$ of $A$ strictly containing $B$, this is equivalent to 
$$ \alpha (W) \cap D \, \subseteq \, \alpha (W) \cap B \, = \, 0,$$
which contradicts the maximality of $B$. Thus $\overline{\alpha}(W)$ is indeed essential in $A/B$. Equivalently, since $W$ is simple,  $A/B$ is uniform, so the exact sequence (\ref{that}) shows that 
\begin{equation}\label{gotit} \mathrm{Ext}^1_S(U/\pi(B),  W) \neq 0.
\end{equation}
But the module $U/\pi(B)$ is finite dimensional by Lemma \ref{typee}, using (\ref{huh}) and since $z$ is a type $(\mathbf{C})$  irreducible element. This means that (\ref{gotit}) contradicts Lemma \ref{split}. Therefore the module $A$ in (\ref{this}) is uniform, as required.
\end{proof}

To state Theorem \ref{gather} precisely it is necessary to consider, here and in Theorem \ref{target}, one-sided versions of property $(\diamond)$ - thus we say that an algebra $A$ satisfies \emph{right} $(\diamond)$ if the injective hulls of its simple right modules are locally Artinian; and similarly for left $(\diamond)$. 
\begin{theorem}\label{gather}\begin{enumerate}
\item[(1)] Let $w$ and $z$ be irreducible elements of $ S$, with
$w$ of type {\bf (B)} and $z$ of type {\bf (C)}. Then the following are equivalent:
\begin{enumerate}
\item[(a)] $\mathrm{Ext}^1_{ S}( S/z S,  S/w S) = 0$;
\item[(b)] $\mathrm{Ext}^1_{ S}( S/ Sw,  S/ Sz) = 0$;
\item[(c)] $\mathrm{Ext}^1_T(T/zT, T/wT) = 0$;
\item[(d)] $\mathrm{Ext}^1_T(T/Tw, T/Tz) = 0$.
\end{enumerate}
\item[(2)] $ S$ satisfies right $(\diamond)$ if and only if the equivalent statements in (1) hold for all such elements $z$ and $w$.
\end{enumerate}
\end{theorem}
\begin{proof}(1) That (a) is equivalent to (b) and (c) is equivalent to (d) is immediate from Proposition \ref{PP}.

To see that (a)$\Longleftrightarrow$(c), set $W :=  S/w S$ and $U :=  S/z S$, so that $W = T/wT$ and $E_{ S}(W) = E_T(W)$ by Proposition \ref{samehull}(2) and Theorem \ref{Bavula}(2). Suppose first that $\mathrm{Ext}^1_{ S}(U,W) \neq 0$. Then by Lemma \ref{patch} there exists a cyclic uniform $ S-$module $V$ which is an extension of $W$ by $U$. Therefore, by Lemma \ref{intersect}(2) and Proposition \ref{samehull}(1), $V \otimes_{ S} T$ is a cyclic uniform $T$-module, which is
an essential extension of its simple $T$-submodule $W \otimes_{ S} T = W$, where the equality follows from Lemma \ref{both}.  Exactness
of the functor $- \otimes_{ S} T$
 and the fact that $U$ is $\mathcal{X}-$torsion free guarantee that
$$ (V \otimes_S T)/ W \cong T/zT \neq 0.$$
This proves that $\mathrm{Ext}^1_T(T/zT,T/wT) \neq 0$, so that (c)$\Longrightarrow$(a). The converse is proved by a similar but easier argument.

\medskip

\noindent (2) Suppose that right $(\diamond)$ holds for $ S$. This means that, for all type {\bf (B)} irreducible elements $w$ of $ S$, $E_{ S}( S/w S)$ is locally artinian. Transferring this fact to the $T-$modules $E_T(T/wT)$ by means of Theorem
\ref{Bavula} and Proposition \ref{samehull}(2), we see that none of these
injective $T-$modules $E_T(T/wT)$ contains a simple $T$-subfactor $T/zT$ with $z$ of type {\bf (C)}, proving (c).

Conversely, if right $(\diamond)$ fails for $ S$ then by Corollary \ref{yip} there
exists a faithful simple right $ S-$module $V$ for which, in the notation of Proposition \ref{torsion}, $W \subsetneq E_{ S}(V)$. In view of Theorem \ref{onecrit} and Proposition \ref{torsion}(5)(ii) there exists a type $\mathrm{(\bf{C})}$ irreducible $d$ of ${ S}$ such that ${ S}/d{ S}\subseteq E_{ S}(V)/W$. Take $e\in E_{ S}(V)$ such that $(e{ S}+W)/W$ is isomorphic to ${ S}/d{ S}$. Then $eS$ is a uniform extension of $V$, and $eS \cap W$ has a finite composition series whose factors are all faithful simple $S$-modules, by Proposition \ref{torsion}(4). Since $eS/eS \cap W \cong S/dS$, an application of the long exact sequence of $\mathrm{Ext}$ shows that
$$ \mathrm{Ext}^1_S(S/dS, Y) \; \neq\; 0 $$
for some composition factor $Y$ of $eS \cap W$. Hence (a) fails to hold, as required.

\end{proof}

\medskip

The following consequence of the above result reduces the issue of the validity of right $(\diamond)$ for $ S$ to a ``monoidal commutativity" condition on the irreducible elements of $T$, as in $(e)$.

\begin{theorem}\label{target} The following statements are equivalent.
\begin{enumerate}
\item[(a)] Right $(\diamond)$ holds for $ S$.
\item[(b)] There does not exist a uniserial right $T$-module with composition length 2 whose socle is a type {\bf (B)} simple and whose simple image is type {\bf (C)}.
\item[(c)] There does not exist a uniserial left $T$-module with composition length 2 whose socle is a type {\bf (C)} simple and whose simple image is type {\bf (B)}.
\item[(d)] There does not exist a cyclic 1-critical left $ S-$module with an infinite dimensional simple image.
\item[(e)] Let $b, c \in  S$ be irreducible, with $b$ type {\bf (B)} and $c$ type {\bf (C)}. Then there exist $b', c'$ in $ S$, respectively irreducibles of types {\bf (B)} and {\bf (C)} and respectively similar to $b$ and to $c$, with
$$cb \; = \; b' c'.$$
\end{enumerate}
\end{theorem}
\begin{proof} $(a) \Longleftrightarrow (b)$: This follows from Theorem \ref{gather}(2) with (c) of Theorem \ref{gather}(1).

\medskip

\noindent $(b)\Longleftrightarrow (c)$: This is the equivalence of Theorem \ref{gather}(c) and (d).

\medskip

\noindent $(c) \Longrightarrow (d)$: Suppose that (c) holds, but that a left $ S$-module $M$ of the form described in (d) exists. That is, $M$ is cyclic and uniform, and is the extension of a 1-critical $ S$-module $A$ by an infinite dimensional simple $ S$-module. Notice that $A$ is faithful: for, if this is not the case $A$ is annihilated by $X^i \theta^j  S$ for some $i,j \geq 0$, and since this ideal satisfies the Artin-Rees property by \cite[Proposition 4.2.6]{McCR}, $(X^i \theta^j S)^tM = 0$ for some $t \geq 1$. This is impossible since $M$ has a faithful factor by hypothesis. Hence, by (the left-side version of) Theorem \ref{onecrit}(1)(b), $A$ contains a copy of $ S/ Sz$ for a type {\bf(C)} irreducible element $z$ of $ S$. Now $M$ is $\mathcal{X}$-torsion free, since its essential submodule $A$ is, and so $T \otimes_{ S} M$ is an essential extension of $T/Tz$ and maps onto $T/Tw$ for a type {\bf(B)} simple critical element $w$ of $ S$. Therefore, since $T \otimes_{ S} M$ has finite composition length and cannot contain any subfactor isomorphic to $T/T\theta$ thanks to Corollary \ref{heck}, an application of the long exact sequence of $\mathrm{Ext}$ shows that it must contain as a subfactor a non-split extension of a type {\bf(C)} simple module by a type {\bf(B)} one, contradicting (c).

\medskip

\noindent $(d) \Longrightarrow (c)$: Conversely, suppose that (d) holds, but that a left $T$-module $N$ exists as in (c), so $N = Ty$ is a non-split $T$-module of composition length two, with unique composition series
$$ 0 \longrightarrow F \longrightarrow N \longrightarrow P  \longrightarrow 0,$$
with $F$ and $P$ respectively type {\bf (C)} and type {\bf (B)} simple $T$-modules. More precisely, let $p \in P$ be such that the epimorphism $\pi:N \longrightarrow P$  maps $y$ to $p$, and by Theorem \ref{Bavula}(1),(2), 
$$ P \; = \;  Sp \;\cong \;   S/ Sb $$
for a type {\bf (B)} irreducible element $b$ of $ S$. By restriction, we obtain an exact sequence
\begin{equation}\label{key2} 0 \longrightarrow F_0 \longrightarrow  Sy \longrightarrow  S/ Sb \longrightarrow 0,
\end{equation}
where $F_0 := F \cap  Sy$. Since $N$ is a uniform $T$-module, it's easy to show that $Sy$ is a uniform $ S$-module, and since $b$ is type {\bf (B)} irreducible, $ S/ Sb$ is a simple $ S$-module. Consider now any non-zero element $u$ of $F_0$. Because $F$ is  a simple $T$-module,
\begin{equation}\label{heart}  T \otimes_{ S}  Su \; = \; Tu \; = \; F,
\end{equation}
so $cu = 0$ for a type {\bf (C)} irreducible element $c$ of $ S$. Thus $F_0$ is hull-similar to its type {\bf (C)} 1-critical submodule
$$C \; := \;  Su\; \cong \;  S/ Sc .$$
Now $F/ Su$ is $\mathcal{X}$-torsion by (\ref{heart}). Therefore every element of the finitely generated $S$-module $F_0/ Su$ is killed by both a power of $X$ and by an irreducible element of type {\bf (C)}. Filtering $F_0/Su$ by the elements killed by $X$, then the elements killed by $X^2$ and so on, we therefore see that $F_0/ Su$ is a finite dimensional $ S$-module. From the exact sequence (\ref{key2}) we obtain the exact sequence
\begin{equation}\label{key3} 0 \longrightarrow F_0/ Su \longrightarrow  Sy/ Su \longrightarrow  S/ Sb \longrightarrow 0.
\end{equation}
By Lemma \ref{torsion1} the $S$-injective hull of $F_0/ Su$ is locally finite dimensional. Since $S/Sb$ is an infinite dimensional simple $S$-module, $Sy/Su$ does not embed in $E_S(F_0/Su)$ and the sequence is split exact. We may thus choose an element $w \in Sy$ whose image \emph{modulo} $ Su$ generates a $ S$-module isomorphic to $ S/ Sb$. Set $W:=  Sw$. We have thus found a non-split exact sequence
\begin{equation}\label{key4} 0 \longrightarrow F_0 \cap W \longrightarrow W \longrightarrow  S/ Sb \longrightarrow 0,
\end{equation}
contradicting (d) and so completing the proof.

\medskip

\noindent $(c) \Longrightarrow (e)$: Assume (c), and let $b$ and $c$ be irreducibles in $ S$, respectively of types $(\mathbf{B})$ and $(\mathbf{C})$. Then $T/Tcb$ has composition length 2, so it is either semisimple or has essential socle. By (c), $T/Tcb$ splits; that is,
$$ T/Tcb \quad  = \quad A \oplus V, $$
where $A = Tb/Tcb \cong T/Tc$ and $V \cong T/Tb$ via the $T-$module projection $\psi: T/Tcb \longrightarrow V$ with $\mathrm{ker}\psi = A$. Set $\psi(1 + Tcb) := v$, so there exists $a \in A$ such that $1 + Tcb =
(a,v)$. Let $M = T(a,v)$.

Since $A = Ta$, there exists a type $(\mathbf{C})$ irreducible element $c'$, which we can choose to be in $ S$ by Lemma \ref{irreducible}, with
$$ \mathrm{Ann}_T(a) = Tc'.$$
Note that $c'$ is similar to $c$ by Theorem \ref{onecrit}(5), because
$$T/Tc' \cong A \cong  T/Tc.$$
Now $c'(a,v) = (0,v')$, and $v' \neq 0$ because $M = T(a,v)$ cannot be a factor of the simple module $T/Tc'$. Thus
$$\mathrm{Ann}_T(v') = Tb''$$
for an irreducible element $b''$, again chosen to be in $ S$ thanks to Lemma \ref{irreducible}, and which is similar to $b$ by Theorem \ref{Bavula}(3). Therefore
$$  Tb''c' \subseteq \mathrm{Ann}_T((a,v)) = \mathrm{Ann}_T(1 + Tcb) =
Tcb,$$
so that
$$Tb''c' = Tcb$$
since the factors by both these left ideals have composition length 2. By
Lemma \ref{intersect}(1) and Remark \ref{extend},
$$  Sb''c' = Tb''c' \cap  S = Tcb \cap  S =  Scb.$$
There is therefore a unit $u$ of $ S$ such that
$$u b''c' = cb.$$
Since $u= 1 + Xw$ for some $w \in  R$, $b' :=ub''$ is a type $(\mathbf{B})$ irreducible in $ S$, similar to $b''$ which is similar to $b$, and so (e) is proved.

\medskip

\noindent $(e) \Longrightarrow (c)$: Assume (e), and let $M$ be as in (c), namely a uniserial left $T-$module which is an extension of a type $(\mathbf{C})$ simple $T-$module $A$ by a type $(\mathbf{B})$ simple module. Being uniserial of length 2, $M$ is cyclic, say $M = Tm$. By hypothesis, $\mathrm{Ann}_T(m + A) = Tb$ for some irreducible type $(\mathbf{B})$ element, which we can choose to be in $ S$ by Lemma \ref{irreducible}. Since $M$ is not simple, $0 \neq bm \in A$, so $A = Tbm$ and there is a type $(\mathbf{C})$ irreducible element $c$ of $ S$ such that $c(bm) = 0$. Therefore $M$ is a factor of $T/Tcb$, and, comparing composition lengths, we see  that
$$ M \cong T/Tcb. $$
Now apply (e) to obtain irreducibles $b'$ and $c'$ in $ S$, respectively of types $(\mathbf{B})$ and $(\mathbf{C})$, with $cb = b'c'$. So
$$ b'c'm \quad = \quad cbm \quad = \quad 0,$$
but $c'm \neq 0$ as $M = Tm$ is not a simple $T-$module. Thus
$$ T(c'm) \cong T/Tb' $$
is a simple type $(\mathbf{B})$ $T-$submodule of $M$, contradicting the fact that $M$ is uniserial with type $(\mathbf{C})$ socle. Therefore no such uniserial module $M$ exists, and (c) is proved.
\end{proof}




\bigskip

\subsection{Proof of $(\diamond)$ for $\widehat{S}$}\label{outcome}

Our aim is to prove that $\widehat{S}=k[[X]][\theta;\alpha]$ satisfies $(\diamond)$ by showing that it satisfies the monoidal commutativity criterion of Theorem \ref{target}(e). For this, we need some additional information about the irreducibility of polynomials in the skew polynomial ring $\widehat{S}$, as provided by the next lemma. Note that in it we do not need to assume that $q$ is not a root of unity and the observation it contains seems to be of independent interest even when $q=1$, that is when $\widehat{S}$ is a commutative ring. We write $q:=\sum_{p\geq 0} q_pX^p$, where $q_p \in k$ and $q_0\neq 0$ as $q$ is a unit of $k[[X]]$. Consider $z=\sum_{i=0}^m z_i\theta^i\in \widehat{S}$
with $z_i \in k[[X]]$ for all $i$. If $X$ divides all the $z_i$ except for $z_0$ and $X^2$ does not divide $z_m$ then the generalised Eisenstein criterion, \cite[\S1, Theorem]{Kovacic}, tells us that $z$ is irreducible. The following lemma is in the direction of a converse to this. Its formulation is rather technical due to the requirements of its proof, but the underlying idea is quite simple. It shows that if there is an integer $n$ with $0< n< m$ such that $X$ does not divide $z_n$ and $X$ divides $z_i$ for all $i>n$, then $z$ is necessarily reducible.

\begin{lemma}\label{notirr} Retain the notation and definition of $\widehat{S}$ as above. Let $1\leq n< m$ and let
$$z=\sum_{i=0}^nf_i\theta^i+\sum_{i=n+1}^m g_i\theta^i\in \widehat{S}, $$
where $f_i, \, g_i \in k[[X]]$ and $g_m \neq 0$. Suppose that
\begin{enumerate}
\item[$(\bullet)$] $f_n$ is invertible (that is, $X\nmid f_n$), and
\item[$(\bullet)$] $g_i$ is not invertible (that is, $g_i \in Xk[[X]]$), for all $i = n+1, \ldots , m$.
\end{enumerate}
Then  $z$ is divisible on the right by a monic polynomial in $\theta$ of degree $n$. In particular, $z$ is a  reducible polynomial.
\end{lemma}
\begin{proof}
Let $z\in \widehat{S}$ be as above. As $ f_n$ is invertible in $\widehat{S}$, we can without loss of generality replace $z$ by $f_n^{-1}z$ and thus assume henceforth that $f_n=1$. We claim that there exist $h_0,\ldots, h_{n-1}\in k[[X]]$ such that
\begin{equation}\label{aim} z\in \widehat{S}(h_0+h_1\theta+\cdots+h_{n-1}\theta^{n-1}-\theta^n).
\end{equation}
Note that the lemma is immediate from (\ref{aim}).

Let $h_0, h_1, \ldots , h_{n-1}$ be elements of $k[[X]]$ which remain to be determined, and denote the left ideal $\widehat{S}(h_0+h_1\theta+\ldots+h_{n-1}\theta^{n-1}-\theta^n)$ of $\widehat{S}$ by $I$. Throughout this proof we denote congruence (mod $I$) by $\equiv$. Thus we are aiming to choose elements $h_i$ of $k[[X]]$ such that
\begin{equation}\label{done} z \in I.
\end{equation}
By definition of $I$,
\begin{equation}\label{key1}
\theta^n \equiv h_0+h_1\theta+\cdots+h_{n-1}\theta^{n-1}.
\end{equation}
Multiplying (\ref{key1}) on the left by $\theta$,  we deduce that
$$\begin{aligned} \theta^{n+1} \; &\equiv& \; \theta\left(\sum_{i=0}^{n-1} h_i \theta^i \right) \qquad \qquad \qquad\\
&\equiv& \; \sum_{i=0}^{n-2} \alpha(h_i)\theta^{i+1} \;+\; \alpha(h_{n-1})\theta^n \qquad \qquad\\
&\equiv& \; \sum_{j=1}^{n-1}\alpha(h_{j-1})\theta^j \;+\; \alpha(h_{n-1})\left(\sum_{j=0}^{n-1} h_{j}\theta^{j}\right)\\
&\equiv& \; \alpha(h_{n-1})h_0 \;+ \;\sum_{j=1}^{n-1}\left( \alpha(h_{j-1}) + \alpha(h_{n-1})h_j \right)\theta^j.
\end{aligned}$$
This yields by induction: for all $\ell \geq 1$,
\begin{equation}\label{hot}
\theta^{n+\ell} \; \equiv \; \sum_{i=0}^{n-1} y_{i, \ell} \theta^i ,
\end{equation}
where for each $i$ and $\ell$, $y_{i,\ell}\in k[[X]]$ is a $k$-linear combination of products of elements of the set $\{\alpha^r(h_j):j\in\{0,\ldots,n-1\}, r\in\{0,\ldots,\ell\}\}$.

Write $h_i=\sum_{p\geq 0} h_{i,p}X^p$ for $h_{i,p}\in k$ and similarly present the elements $f_i, g_j$ and $ y_{i,j-n}$ of $k[[X]]$. Note that $\alpha(h_j)=\sum_p h_{j,p}(qX)^p$, for any $r\in\N$, $\alpha^r (h_j)=\sum_p h_{j,p}(\alpha^r (X))^p$ and that $\alpha^r(X)=N_r^{\alpha}(q)X$, so $\alpha^r(h_j)_p$ is a sum of products of $h_{j,p*}$ for $p*\leq p$ with some coefficients of $q$, namely $q_p*$ for $p*\leq p$. This means that, for all $r,p \in \mathbb{N}$ and $j \in \{0, \ldots , n-1\}$,
\begin{equation}\label{key}
\mbox{$\alpha^r(h_j)_p$ \emph{is an explicit} $k$-\emph{linear combination of} $h_{j,p*}$ \emph{for} $p*\leq p$}.
\end{equation}
(In case $q\in k$, we actually have $\alpha^r(h_j)_p=q^rh_{j,p}$, a $k$-linear combination of $h_{j,p}$.)

Since $f_n = 1$, from (\ref{key1}) and (\ref{hot}) we have
\begin{equation}\label{division}
z\equiv \sum_{i=0}^{n-1}f_i\theta^i+\sum_{i=0}^{n-1}h_i\theta^i+
\sum_{j=n+1}^{m}\sum_{i=0}^{n-1} g_{j}y_{i,j-n}\theta^i.
\end{equation}
So $z\in \widehat{S}(h_0+h_1\theta+\ldots+h_{n-1}\theta^{n-1}-\theta^n)$
if and only if  for each $i\in\{0,\ldots, n-1\}$
\begin{equation}\label{system of quations}
f_i+h_i+\sum_{j=n+1}^{m} g_{j}y_{i,j-n}=0
\end{equation}
in $k[[X]]$.

We claim that the system of equations (\ref{system of quations}) in the unknowns $h_0,\ldots, h_{n-1}$ does indeed have a solution in $k[[X]]$.

Recall that $X\mid g_j$ for all $j\in\{n+1,\ldots,m\}$. Therefore, looking at
coefficients of terms of degree $p = 0$ and $p=1$ in $X$ in (\ref{system of quations}) , we get
\begin{eqnarray}
&h_{i,0}=-f_{i,0}, & \forall i\in\{0,\ldots,n-1\} \\
&h_{i,1}=-f_{i,1}-\sum_{j=n+1}^{m} g_{j,1}y_{i,j-n,0}, &\forall i\in\{0,\ldots,n-1\}
\end{eqnarray}
where, by (\ref{hot}) and (\ref{key}), $y_{i,j-n,0}$ is a $k$-linear combination of products of elements from the set
$$ \{\alpha^r(h_{j,0}):j\in\{0,\ldots,n-1\}, r\in\{0,\ldots,m-n\}\}.$$
Assume now that $t$ is a positive integer and that $h_{i,p}$ are known for all $i\in\{0,\ldots,n-1\}$ and $p\leq t$. By (\ref{system of quations}), we require, for all $i\in\{0,\ldots,n-1\}$,
\begin{equation}
h_{i,t+1}=-f_{i,t+1}-\sum_{j=n+1}^{m}\sum_{s=0}^{t} g_{j,(t+1-s)}y_{i,j-n,s}
\end{equation}
where, by (\ref{hot}) and (\ref{key}), for each $i\in \{0,\ldots,n-1\}$ and $s\in \{0,\ldots, t\}$, $y_{i,j-n,s}$ is a $k$-linear combination of products of elements from
$$ \{\alpha^r(h_{j,\ell}):j\in\{0,\ldots,n-1\}, r\in\{0,\dots,i\}, \ell\in\{0,\ldots,s\}\}.$$
So each $h_{i,t+1}$ exists, and hence $h_0,\ldots, h_{n-1}$ exist also.
\end{proof}

The following example shows that the above lemma does not hold when the base ring is replaced by $k[X]_{\langle X \rangle}$.

\begin{ex} \normalfont For $q\in k^*$, let $S=k[X]_{\langle X \rangle}[\theta; \alpha]$ where $\alpha(X)=qX$. The element $r=1+X+\theta+X\theta^2\in S$ is irreducible in $S$. To check this note that if $a,b\in k[X] $ are such that $b\notin <X>$, then  $$r=\alpha(b^{-1})b^{-1}[bX\theta-X\alpha(a)+\alpha(b)](a+b\theta)+1+X-b^{-1}a+Xb^{-1}\alpha(b^{-1})\alpha(a)a.$$So $r\in S(a+b\theta)$ if and only if $\alpha(b)(b-a)=-X(\alpha(b)b+\alpha(a)a)$. Comparing the degrees on the RHS and LHS of the last equation we conclude that $r$ is not divisible on the right by a polynomial of the form $a+b\theta$ for any $a,b\in k[X] $ with $b\notin <X>$.

Making use of the above for $S^{op}=k[X]_{\langle X \rangle}[\theta; \alpha^{-1}]$ we see that also, with the same constraints on $a$ and $b$, $r\notin (a+b\theta)S$, proving that $r$ is irreducible.
\end{ex}

A very useful consequence of Lemma \ref{notirr} provides a canonical form
for each type $(\mathbf{C})$ similarity class of irreducible elements of $\widehat{S}$:

\begin{corollary}\label{typeC} Let $c$ be a type $(\mathbf{C})$ irreducible element of $\widehat{S}$, of degree $n \geq 1$ as a polynomial in $\theta$. Then there is a unit $u$ of $\widehat{S}$ such that $\hat{c} := uc$ is an irreducible element of $\widehat{S}$, so $\hat{c}$ is similar to $c$, with $\hat{c}$ having form
$$ \hat{c} \quad = \quad h_0 + \sum_{i=1}^{n-1}h_i \theta^i + \theta^n, $$
when $n > 1$, with $h_0 \neq 0$ and $h_i \in k[[X]]$ for all $i$; and $\hat{c}$ having form
$$ \hat{c} \quad = \quad h_0 + \theta $$
with $0 \neq h_0 \in k[[X]]$  when $n = 1$.
\end{corollary}

\begin{proof} Let $c = \sum_{i=0}^n r_i \theta^i$, where $r_i \in k[[X]]$ for all $i$ and $r_n \neq 0$. Then
\begin{equation}\label{nzero} r_0 \neq 0,
\end{equation}
since otherwise $c = u\theta$ for some unit $u$ of $\widehat{S}$ as $c$
is irreducible, so that $c$ has type $(\mathbf{A})$, which is a contradiction. Moreover
\begin{equation}\label{nnzero} n > 0,
\end{equation}
because if $n=0$ then $c \in k[[X]]$, which is again a contradiction to
$c$ being type $(\mathbf{C})$.

There exists at least one $i$ for which
$X \nmid r_i$, because otherwise, as $c$ is irreducible,  $c = Xu'$ for
a unit $u'$ of $\widehat{S}$, once more contradicting the fact that $c$ is type $(\mathbf{C})$. Let $m$ be the greatest value of $i$ such that
\begin{equation}\label{nmid}  X \nmid  r_m.
\end{equation}
If $m= 0$ then, by (\ref{nzero}), $r_0$ is a unit in $\widehat{S}$ and so
$$ \widehat{c} \quad := \quad r_0^{-1}c \quad = \quad 1 + X\tau $$
for some $\tau \in \widehat{S}$, so that $c$ is type $(\mathbf{B})$, a contradiction. So
\begin{equation}\label{great} m > 0.
\end{equation}
Suppose that $m < n$. Then
\begin{equation}\label{bad} c \quad =  \sum_{i=0}^m r_i\theta^i + \sum_{i=m+1}^n r_i \theta^i, \qquad r_m \notin \langle X \rangle,\; r_i \in \langle X \rangle \forall i > m,\; r_n \neq 0.
\end{equation}
But then, noting (\ref{great}), Lemma \ref{notirr} applies to the expression (\ref{bad}) for $c$, and tells us that $c$ is reducible, a contradiction. Therefore $m=n$, so that $r_n$ is a unit in $k[[X]]$. That is, we can define
$$\hat{c} := r_n^{-1} c ,$$
so that
\begin{equation}\label{form} \hat{c} \quad = \quad  \sum_{i=0}^{n-1}h_i \theta^i \, + \, \theta^n,
\end{equation}
with all $h_i \in k[[X]]$, $h_0 \neq 0$ and $n >0$. Finally, by $\S$\ref{lrPID}, $\hat{c}$ is similar to $c$.
\end{proof}

\bigskip

We now apply Lemma \ref{notirr} and Corollary \ref{typeC} to prove that the monoidal commutativity condition of Theorem \ref{target}(e) is satisfied, hence deducing that $\widehat{S}$ satisfies $(\diamond)$.

\begin{proposition}\label{yippee} Let $b$ and $c$ be irreducible elements
of $\widehat{S}$, respectively of type $(\mathbf{B})$ and $(\mathbf{C})$.
Then there exist irreducible elements $b'$ and $c'$ of $\widehat{S}$, respectively similar to $b$ and $c$, such that
$$ cb \quad = \quad b'c'.$$
\end{proposition}

\begin{proof} Note that if $c$ is replaced by $\hat{c} := uc$ for a unit $u$ of $\widehat{S}$, and we prove that
$$ \hat{c}b \quad = \quad  \hat{b}c' $$
with $\hat{b}$ and $c'$ irreducibles respectively of types $(\mathbf{B})$
and $(\mathbf{C})$, then
$$ cb = u^{-1}\hat{c}b = (u^{-1}\hat{b})c' = b'c',$$
with $b'$ and $c'$ irreducibles in $\widehat{S}$ respectively of types $(\mathbf{B})$ and $(\mathbf{C})$.

Thus Corollary \ref{typeC} allows us to assume that $c$ has the form
\begin{equation}\label{form} c \quad = \quad  \sum_{i=0}^{n}h_i \theta^i \,
\end{equation}
with $h_i \in k[[X]], \, h_0 \neq 0$, $n >0$ and $h_n=1$.

Let $b = 1 + Xs$, where $s = \sum_{j=0}^{\ell}s_j \theta^j$ with $s_j \in k[[X]]$ for all $j$. Since $b$ is irreducible and so not invertible in $\widehat{S}$,
\begin{equation}\label{large} \ell \; > \; 0.
\end{equation}
Thus, using (\ref{form}),
\begin{align*}  cb \quad &=&  &\left(\sum_{i=0}^{n}h_i \theta^i \right)\left(1 + X(\sum_{j=0}^{\ell}s_j \theta^j)\right)\\
&=&  &\sum_{i=0}^n h_i \theta^i + X\left(\sum_{i=0}^n\sum_{j=0}^{\ell}q^ih_i\alpha^i(s_j)\theta^{i+j}\right).
\end{align*}
In the above expression for $cb$, observe that $n > 0$ and $h_n = 1$ by
(\ref{nnzero}) and (\ref{form}), while the highest power of $\theta$ occurring is $n + \ell$, with $n + \ell > n$ by (\ref{large}). Since the coefficient of $\theta^{n + \ell}$ is divisible by $X$, the hypotheses of Lemma \ref{notirr} are satisfied by $cb$. Therefore we can conclude that $cb$ is divisible on the right by a monic polynomial $c'$ of degree  $n$; that is,
\begin{equation}\label{atlast} cb \quad = \quad b'c'.
\end{equation}
Comparing degrees on the right and left of (\ref{atlast}),
$$ \mathrm{deg}_{\theta}(b') = \ell > 0 $$
by (\ref{large}), so that neither $c'$ nor $b'$ are units in $\widehat{S}$. Comparing the two factorisations provided by (\ref{atlast}) in the PID
$T$, both must have length 2 by the Jordan-Holder theorem, since $T/Tcb =
T/Tb'c'$. Thus, both $b'$ and $c'$ are irreducible in $T$ and hence, by Lemma \ref{irreducible}, in $\widehat{S}$. Moreover, since $c'$ is monic in $\theta$ of degree $n > 0$, it cannot be expressed as $c' = (1 + Xd)u$ for any $d \in \widehat{S}$ and unit $u \in k[[X]]$. Therefore $c'$ has
type $(\mathbf{C})$, and so, again applying the Jordan-Holder theorem to $T/Tcb = T/Tb'c'$, $b'$ has type $(\mathbf{B})$. This then forces $b$ to be similar to $b'$ and $c$ to $c'$, as required.
\end{proof}

The version of the following corollary for the right $(\diamond)$ property is immediate from the equivalence of ($a$) and ($e$) of Theorem \ref{target} coupled with Proposition \ref{yippee}. Since the arguments and results of this paper can be repeated with ``right'' and ``left'' swapped throughout, left $(\diamond)$ follows in a parallel fashion. 
Note that we only need to discuss here the proof of the ``difficult case'' where $\alpha$ is not of finite order, since $k[[X]][\theta; \alpha]$ is a finite module over its centre otherwise.

\begin{corollary}\label{done} Let $k$ be a field and $\alpha$ an arbitrary $k$-automorphism of $k[[X]]$. Then $\widehat{S}=k[[X]][\theta; \alpha]$ satisfies $(\diamond)$.
\end{corollary}

The fact that $(\diamond)$ holds for $\widehat{S}$ enables one to improve
the description of cyclic faithful critical $\widehat{S}-$modules from that given by Theorem \ref{onecrit}(4)(a) and (5), as follows:

\begin{corollary}\label{cyclic} Let $I$ be a left [resp. right] ideal of $\widehat{S}$ with $\widehat{S}/I$ faithful and 1-critical. Then $I = \widehat{S}z$ [resp. $I = z\widehat{S}$] for an irreducible element $z$ of $\widehat{S}$ (which necessarily will have type $(\mathbf{C})$).
\end{corollary}
\begin{proof} We prove the version for left ideals; the right hand version is obtained by interchanging left and right throughout.  Let $M = \widehat{S}/I$ be a cyclic faithful 1-critical left $\widehat{S}-$module. By
the left module version of Theorem \ref{onecrit}(1)(b), $M$ is hull-similar to $\widehat{S}/\widehat{S}c$ for a type ($\mathbf{C}$) irreducible element $c$ of $\widehat{S}$. In particular,

\begin{equation}\label{free} M \textit{ is } \mathcal{X}-\textit{torsion free.}
\end{equation}

The equivalent statements (a) - (e)  of Theorem \ref{target} all hold, by Corollary \ref{done}. By Theorem \ref{target}(d), every proper factor of $M$ is a finite dimensional vector space, and in particular $\mathcal{X}-$torsion. We claim that this implies that
\begin{equation}\label{small} T/TI \textit{ is a simple } T-\textit{module.}
\end{equation}
For, let $N$ be a non-zero $T$-submodule of $T/TI$. Then $N \cap M \neq 0$ by (\ref{free}), so that, by Theorem \ref{target}(d) again,
\begin{equation}\label{size} \mathrm{dim}_k(M/N \cap M) < \infty.
\end{equation}
Let $y \in T \otimes_{\widehat{S}} M$. Thus there exists $n \geq 1$ such that $X^n y \in M$ and then, by (\ref{size}), there exists $t \geq 1$ such that
$$ X^t(X^n y) \in N. $$
But $N$ is a $T-$submodule of $T \otimes_{\widehat{S}} M$, so $y = X^{-(t+n)}X^{t+n}y \in N$. Thus $N = T \otimes_{\widehat{S}} M$ and (\ref{small}) follows.

In view of (\ref{small}), $TI = Tz$ for an irreducible element $z$, which we can choose to be in $\widehat{S}$ by Lemma \ref{irreducible}. That is, $z \in TI \cap \widehat{S} = I$, and $z$ has type ($\mathbf{C}$) since otherwise it has type $(\mathbf{B})$ and then $\widehat{S}/\widehat{S}z$ is simple by the left module version of Theorem \ref{Bavula}(2), contradicting the fact that $\widehat{S}/I$ is 1-critical. Therefore $\widehat{S}/\widehat{S}z$ is itself 1-critical by the left module version of Theorem \ref{onecrit}(2), and hence $\widehat{S}z = I$, as required.

\end{proof}

\section*{Acknowledgements}
Part of this work was done when the second author visited the University of Glasgow in February 2020 supported by a FCT grant
SFRH/BSAB/150375/2019. She would like to thank the Univerity and its staff for their hospitality. This work was also partially supported by CMUP, which is financed by national funds through FCT--Funda\c{c}\~ao para a Ci\^encia e a Tecnologia, I.P., under the project with reference UIDB/00144/2020. The work of the first author was partially supported by Leverhulme Emeritus Fellowship EM-2017-081/9.

\bigskip

  \textsc{School of Mathematics and Statistics, University of Glasgow, Glasgow G12 8SQ, Scotland}\par\nopagebreak
  \textit{E-mail address}: \texttt{Ken.Brown@glasgow.ac.uk}

  \medskip
   \textsc{CMUP, Departamento de Matem\'atica, Faculdade de Ci\^encias, Universidade do Porto,  rua do Campo Alegre s/n, 4169-007 Porto, Portugal}\par\nopagebreak
  \textit{E-mail address}: \texttt{pbcarval@fc.up.pt}

  \medskip

   \textsc{Institute of Mathematics, University of Warsaw,  Banacha 2, 02-097 Warsaw, Poland}
  \par\nopagebreak
  \textit{E-mail address}:   \texttt{jmatczuk@mimuw.edu.pl}

  \medskip


\begin{bibdiv}
 \begin{biblist}
 \bib{Bavula}{article}{
   author={Bavula, V.},
   author={Van Oystaeyen, F.},
   title={The simple modules of certain generalized crossed products},
   journal={J. Algebra},
   volume={194},
   date={1997},
   number={2},
   pages={521--566},
   issn={0021-8693},
   review={\MR{1467166}},
   doi={10.1006/jabr.1997.7038},
}
\bib{BCM}{article}{
   author={Brown, K.A.},
   author={Carvalho, P. A. A. B.},
   author={Matczuk, J.},
   title={Simple modules and their essential extensions for skew polynomial
   rings},
   journal={Math. Z.},
   volume={291},
   date={2019},
   number={3-4},
   pages={877--903},
   issn={0025-5874},
   review={\MR{3936092}},
   doi={10.1007/s00209-018-2128-8},
}


\bib{Cohn_NUFD}{article}{
   author={Cohn, P. M.},
   title={Noncommutative unique factorization domains},
   journal={Trans. Amer. Math. Soc.},
   volume={109},
   date={1963},
   pages={313--331},
   issn={0002-9947},
   review={\MR{155851}},
   doi={10.2307/1993910},
}

\bib{Cohn}{book}{
   author={Cohn, P. M.},
   title={Free Ideal Rings and Localization in General Rings},
   series={New Mathematical Monographs},
   volume={3},
   publisher={Cambridge University Press, Cambridge},
   date={2006},
   pages={xxii+572},
   isbn={978-0-521-85337-8},
   isbn={0-521-85337-0},
   review={\MR{2246388}},
   doi={10.1017/CBO9780511542794},
}

\bib{Eisenbud}{book}{
   author={Eisenbud, D.},
   title={Commutative Algebra with a View toward Algebraic Geometry},
   series={Graduate Texts in Mathematics},
   volume={150},  
   publisher={Springer-Verlag, New York},
   date={1995},
   pages={xvi+785},
   isbn={0-387-94268-8},
   isbn={0-387-94269-6},
   review={\MR{1322960}},
   doi={10.1007/978-1-4612-5350-1},
}


\bib{GW}{book}{
   author={Goodearl, K. R.},
   author={Warfield, R. B., Jr.},
   title={An Introduction to Noncommutative Noetherian Rings},
   series={London Mathematical Society Student Texts},
   volume={61},
   edition={2},
   publisher={Cambridge University Press, Cambridge},
   date={2004},
   pages={xxiv+344},
   isbn={0-521-83687-5},
   isbn={0-521-54537-4},
   review={\MR{2080008}},
   doi={10.1017/CBO9780511841699},
}

 \bib{KaplanskyC}{book}{
   author={Kaplansky, I.},
   title={Commutative Rings},
   edition={Revised edition},
   publisher={The University of Chicago Press, Chicago, Ill.-London},
   date={1974},
   pages={ix+182},
   review={\MR{0345945}},
}

\bib{Kovacic}{article}{
   author={Kovacic, J.},
   title={An Eisenstein criterion for noncommutative polynomials},
   journal={Proc. Amer. Math. Soc.},
   volume={34},
   date={1972},
   pages={25--29},
   issn={0002-9939},
   review={\MR{292803}},
   doi={10.2307/2037888},
}

\bib{Jacobson}{book}{
   author={Jacobson, N.},
   title={The Theory of Rings},
   series={American Mathematical Society Mathematical Surveys, vol. II},
   publisher={American Mathematical Society, New York},
   date={1943},
   pages={vi+150},
   review={\MR{0008601}},
}

\bib{FBN}{article}{
   author={Jategaonkar, A. V.},
   title={Jacobson's conjecture and modules over fully bounded Noetherian
   rings},
   journal={J. Algebra},
   volume={30},
   date={1974},
   pages={103--121},
   issn={0021-8693},
   review={\MR{352170}},
   doi={10.1016/0021-8693(74)90195-1},
}

\bib{JurekLeroy}{article}{
   author={Leroy, A.},
   author={Matczuk, J.},
   title={Primitivity of skew polynomial and skew Laurent polynomial rings},
   journal={Comm. Algebra},
   volume={24},
   date={1996},
   number={7},
   pages={2271--2284},
   issn={0092-7872},
   review={\MR{1390373}},
   doi={10.1080/00927879608825699},
}

\bib{McCR}{book}{
   author={McConnell, J. C.},
   author={Robson, J. C.},
   title={Noncommutative Noetherian Rings},
   series={Graduate Studies in Mathematics},
   volume={30},
   edition={Revised edition},
   note={With the cooperation of L. W. Small},
   publisher={American Mathematical Society, Providence, RI},
   date={2001},
   pages={xx+636},
   isbn={0-8218-2169-5},
   review={\MR{1811901}},
   doi={10.1090/gsm/030},
}


\bib{Matlis}{article}{
   author={Matlis, E.},
   title={Injective modules over noetherian rings},
   journal={Pacific J. Math.},
   volume={8},
   date={1958},
   number={},
   pages={511-528},
   issn={},
   review={\MR{0099360}},
   doi={},
}

\bib{Musson80}{article}{
   author={Musson, I. M.},
   title={Injective modules for group rings of polycyclic groups. II},
   journal={Quart. J. Math. Oxford Ser. (2)},
   volume={31},
   date={1980},
   number={124},
   pages={449--466},
   issn={0033-5606},
   review={\MR{596979}},
   doi={10.1093/qmath/31.4.429},
}

\bib{Ore}{article}{
   author={Ore, O.},
   title={Theory of noncommutative polynomials},
   journal={Annals of Mathematics},
   volume={34},
   date={1933},
   number={3},
   pages={480-508},
}


\bib{PP_99}{article}{
   author={Prest, M.},
   author={Puninski, G.},
   title={Some model theory over hereditary Noetherian domains},
   journal={J. Algebra},
   volume={211},
   date={1999},
   number={1},
   pages={268--297},
   issn={0021-8693},
   review={\MR{1656581}},
   doi={10.1006/jabr.1998.7607},
}

 \end{biblist}
\end{bibdiv}
 \end{document}